\documentclass[12pt,a4paper]{amsart}
\usepackage[ansinew]{inputenc}
\usepackage[T1]{fontenc}
\usepackage[english]{babel}
\usepackage{amsthm}
\usepackage{amsfonts}         
\usepackage{amsmath}
\usepackage{comment}
\usepackage{amssymb}
\usepackage{empheq}

\usepackage{graphicx}
\usepackage{psfrag}

\usepackage{url} 
\usepackage{framed}
\usepackage[T1]{fontenc}
\usepackage{amsmath}
\usepackage{amssymb}
\usepackage[usenames]{color}
\usepackage{amssymb}
\usepackage{latexsym}

\newcommand{\altomega}{\sigma}
\newcommand{\mtext}{}
\newcommand{\mmtext}{}
\newcommand{\matti}{}

\newcommand{\removedEinstein}[1]{}

\newcommand{\extension}[1]{} 

\newcommand{\generalizations}[1]{}

\newcommand{\HOX}[1]{}

\newcommand{\hiddenfootnote}[1]{}

\def\cN{{\mathcal N}}

\newcommand{\R}{{\mathbb R}} 
\newcommand{\D}{{\cal D}}

\renewcommand{\D}{{\mathcal D}}  

\renewcommand{\L}{{\mathcal L}}

\def\hat{\widehat}
\def\tilde{\widetilde}
\def \bfo {\begin {eqnarray*} }
\def \efo {\end {eqnarray*} }
\def \ba {\begin {eqnarray*} }
\def \ea {\end {eqnarray*} }
\def \beq {\begin {eqnarray}}
\def \eeq {\end {eqnarray}}
\def \supp {\hbox{supp}\,}
\def \dim{\hbox{dim}\,}

\def \diam {\hbox{diam }}

\def \dist {\hbox{dist}}

\def \det {\hbox{det}}

\def \e {\varepsilon}
\def \p {\partial}
\def \a {\alpha}
\renewcommand{\b}{\beta}

\def\Z{{\mathbb Z}}

\newtheorem{definition}{Definition}[section]

\newtheorem{proposition}[definition]{Proposition} 
 
\newtheorem{remark}[definition]{Remark}

\theoremstyle{plain}
\newtheorem{Df}[definition]{Definition}
\newtheorem{Le}[definition]{Lemma}
\newtheorem{Th}[definition]{Theorem}
\newtheorem{Po}[definition]{Proposition}

\theoremstyle{plain}
\begin{document}
\newcommand{\QED}{\begin{flushright}
$\square $
\end{flushright}}

\title[Distance difference functions]
{Determination of a Riemannian manifold from the distance difference functions}
\author{Matti Lassas and Teemu Saksala}
\address{Matti Lassas and Teemu Saksala, Department of Mathematics and Statistics, University of Helsinki, Finland}
\date{\today}
\maketitle

\begin{abstract}
Let $(N,g)$ be a Riemannian manifold with the distance function $d(x,y)$ and an open subset $M\subset N$. For $x\in M$ we denote by $D_x$ the distance difference function $D_x:F\times F\to \mathbb R$, given by $D_x(z_1,z_2)=d(x,z_1)-d(x,z_2)$, $z_1,z_2\in F=N\setminus M$. We consider the inverse problem of determining the topological and the differentiable structure of the manifold $M$ and the metric $g|_M$ on it when we are given the distance difference data, that is, the set $F$, the metric $g|_F$, and the collection $\mathcal D(M)=\{D_x;\ x\in M\}$. Moreover, we consider the embedded image $\mathcal D(M)$ of the manifold $M$, in the vector space $C(F\times F)$, as a representation of manifold $M$. The inverse problem of determining $(M,g)$ from $\mathcal D(M)$ arises e.g. in the study of the wave equation on $\mathbb R\times N$ when we observe in $F$ the waves produced by spontaneous point sources at unknown points $(t,x)\in \mathbb R\times M$. Then $D_x(z_1,z_2)$ is the difference of the times when one observes at points $z_1$ and $z_2$ the wave produced by a point source at $x$ that goes off at an unknown time. The problem has applications in hybrid inverse problems and in geophysical imaging.
\end{abstract}


\tableofcontents
\section{Introduction}

\subsection{Motivation of the problem}
Let us consider a body in which there  \matti{spontaneously appear} point sources that create propagating waves. 
{In various applications one encounters a geometric 
 inverse problem} where we detect such waves either outside or at the boundary of the body and aim to determine the unknown wave speed inside the body. 
\matti{As an example of such situation, one can consider the micro-earthquakes that appear very frequently 
near  active faults. The related inverse problem is whether the surface observations
of elastic waves produced by the micro-earthquakes can be used in the geophysical imaging
of Earth's subsurface \cite{Kayal,Sava}, that is, to  determine the speed of the elastic waves in the studied volume.} 
In this paper we consider a highly idealized version of the above inverse problem:
We consider the problem on an $n$  dimensional manifold $N$ with a Riemannian metric $g$. The distance function determined by this metric tensor
corresponds physically to the travel time of a wave between two points. The Riemannian distance
of points $x,y\in N$ is denoted by $d(x,y)$. For simplicity
we assume that the manifold $N$  is compact and has no boundary.
Instead of considering measurements on boundary, we assume that the manifold
contains an unknown part $M\subset N$ and the metric is known outside {the set $M$}.
When \matti{a spontaneous point source} produces a wave at some unknown point $x\in M$
at some unknown time $t\in \R$, the produced wave is observed at the point
$z\in N\setminus M$  at time $T_{x,t}(z)=d(z,x)+t$. These observation times
at two points $z_1,z_2\in N\setminus M$ determine
the {\it distance difference function}
\beq\label{def: dist diff}
D_x(z_1,z_2)=T_{x,t}(z_1)-T_{x,t}(z_2)=d(z_1,x)-d(z_2,x).
\eeq
Physically, this function corresponds to the difference of times
at $z_1$ and $z_2$ of the waves produced by the point source at $(x,t)$, {see Fig 1.\ and
Section \ref{subset: application}.}
The assumption that there is a large number point sources and that we do measurements over
a long time can be modeled by the assumption that we are given \matti{the set $N\setminus M$ and}
the family of functions
\ba
\{D_x\ ;\ x\in X\}\subset  C((N\setminus M)\times (N\setminus M)),
\ea
where $X\subset M$ is either the whole manifold $M$ or its dense subset, 
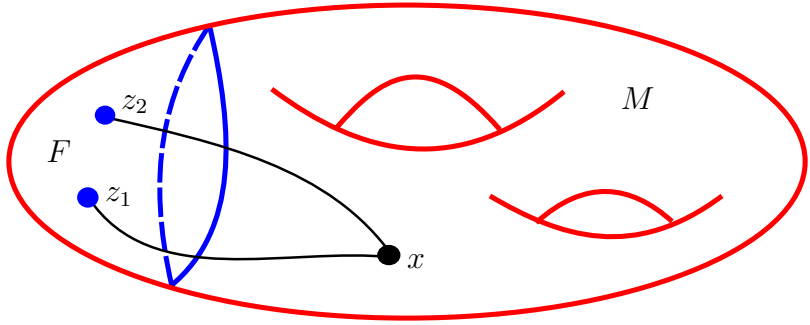
\begin{figure}
 \begin{picture}(100,150)
  \put(-100,10){\includegraphics[width=300pt]{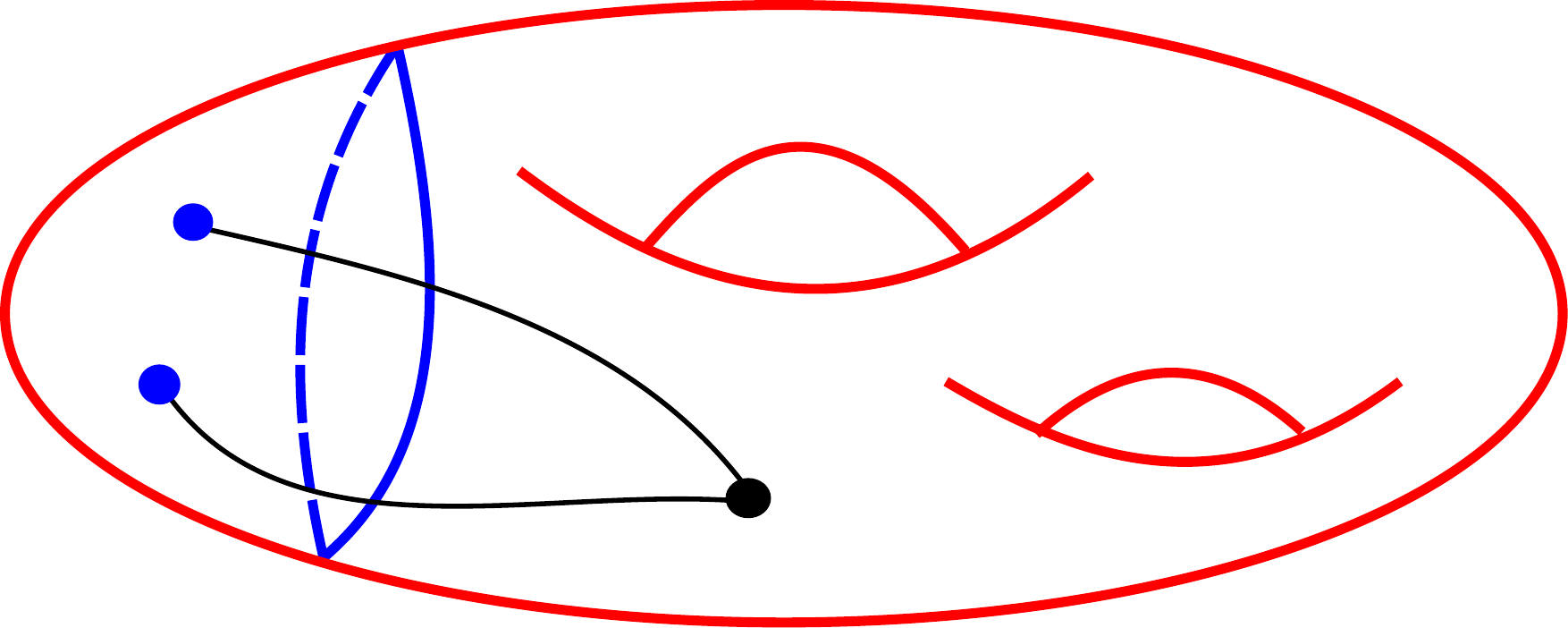}}
  \put(50,30){$x$}
  \put(-63,55){$z_1$}
  \put(-57,90){$z_2$}
  \put(-85,70){$F$}
  \put(130,90){$M$}
  
   \end{picture}
\caption{\it The distance difference functions are related to observation on the
closed manifold $N$ that contains an unknown open subset $M$ and its known
complement $F=N\setminus M$. The distance difference function $D_x$
associated to a source point $x\in M$ has,  at the observation points $z_1,z_2\in F$,
the value $D_x(z_1,z_2)=d(x,z_1)-d(x,z_2)$. Consider the wave equation and 
a wave that is  produced by a point source at $x$ that goes off at an unknown time and that 
is observed on $F$. Then
the difference of the times when the wave is
observed at the points $z_1$ and $z_2$ is equal to  $D_x(z_1,z_2)$. The time difference
inverse problem is to determine the topology and the isometry type of $(N,g)$ from
such observations {when $x$  runs over a dense subset of $M$}.}
\end{figure}
\matti{see Remark \ref{rem: density} below.}

\subsection{Definitions and the main result}
\label{The problem setting of this paper} 

Let $(N_1,g_1)$ and $(N_2,g_2)$ be compact and connected Riemannian manifolds without boundary. 
Let $d_j(x,y)$ denote the Riemannian distance of points $x,y\in N_j$, $j=1,2$. 
Let $M_j\subset N_j$ 
be open sets and define closed sets $F_j=N_j \setminus M_j$. Suppose $F^{int}_j\neq \emptyset$. This is a crucial assumption and we provide a counterexample for a case $F^{int}_j= \emptyset$ in Appendix A.

\matti{Below, we assume that we know $F_j$ as a differentiable manifold, that is,} 
we know the atlas of $C^\infty$-smooth coordinates on $F_j$, and the metric tensor $g_j|_{F_j}$ on $F_j$,
but we do not know the manifold $(M_j,g_j|_{M_j})$. We assume $F_j$ to be a smooth manifold with smooth boundary $\partial F_j=\partial M_j$.

\begin{definition}\label{def: dd data}
\matti{For $j=1,2$ and all points $x\in N_j$ we define the distance difference function
$$
D^j_{x}: F_j \times F_j \rightarrow \R, \: D^j_{x}(z_1,z_2):=d_j(z_1,x)-d_j(z_2,x)
$$
where $F_j=N_j\setminus M_j$.
Recall that here $d_j$ is the Riemannian distance function of manifold $N_j$. 
\matti{We denote by
\ba
\D_j:N_j\to C(F_j \times F_j),\quad \D_j(x)=D_x^j
\ea
the map from a point $x$ to the corresponding distance difference function $D_x^j$}.
The pair
 $(F_j,g_j|_{F_j})$ and the collection 
\ba
\D_j(M_j)=\{D^j_x\,;\ x\in M_j\}\subset C(F_j\times F_j)
\ea of the distance difference functions of the points $x\in M_j$
is called the distance difference data for the set $M_j$.}
\end{definition}

 We emphasize that the above collections
$\{D^j_x(\cdot,\cdot);\ x\in M_j\}$ are given as unindexed subsets of $C(F_j\times F_j)$,
that is, for a given element  $D^j_x(\cdot,\cdot)$ of this set we do not know what
is the corresponding ``index point''  $x$.

 To prove
the uniqueness of this inverse problem, we assume the following:
\beq
& &
\hbox{There is a diffeomorphism  $\phi:F_1 \rightarrow F_2$
such that $\phi^\ast g_2|_{F_2}=g_1|_{F_1}$},  \label{Data1a}
\\
& &
\{D^1_x(\cdot,\cdot)\ ; \ x\in M_1\}=\{D^2_y(\phi(\cdot),\phi(\cdot))\ ;\ y \in M_2\}.
\label{Data1b}
\eeq


\matti{The following proposition states that using the small data $\D_j(M_j)$ we can
 construct the  bigger data set $\D_j(N_j)$.}
\begin{Po}\label{bigger data}
Assume that \eqref{Data1a}-\eqref{Data1b} are valid. Then:
\begin{enumerate}
\item  [(i)] The map $\phi:F_1 \rightarrow F_2,$ is a metric isometry, that is, $d_1(z,w)=d_2(\phi(z),\phi(w))$ for all $z,w\in F_1$.
\item  [(ii)] The collections $\D_j(N_j)=
\{D^j_x(\cdot,\cdot);\ x\in N_j\}\subset C(F_j\times F_j)$ are equivalent in the following sense 
\beq\label{Data}
\{D^1_x(\cdot,\cdot)\ ;\ x\in N_1\}=\{D^2_y(\phi(\cdot),\phi(\cdot))\ ;\ y \in N_2\}.
\eeq
\end{enumerate}
\end{Po}

We postpone the proof of this proposition and the other results in the introduction
and give the proofs later in the paper.

The main theorem of the paper is the following:
\begin{Th}
Let $(N_1,g_1)$ and $(N_2,g_2)$ be compact and connected Riemannian manifolds, without boundary, 
{of dimension $n\geq 2$.} 
Let $M_j\subset N_j$ be open sets and define closed sets $F_j=N_j \setminus M_j$. Assume that $F^{int}_j\neq \emptyset$ and that $F_j$ is a smooth manifold with smooth boundary $\partial F_j=\partial M_j$. Also, suppose 
that  assumptions \eqref{Data1a}-\eqref{Data1b} are valid, that is,  the distance difference data  for sets $M_1$  and $M_2$ are equivalent. 
\matti{Then the
manifolds $(N_1,g_1)$ and $(N_2,g_2)$ are isometric.}
\label{main theorem}
\end{Th}

We prove Theorem \ref{main theorem} in {Section \ref{sec: proof of main result}}. This proof is divided into 5 {subsections}. In the first we set notations and consider some basic facts about geodesics. In the second we prove Proposition \ref{bigger data}. In the third we show that manifolds $(N_j, g_j)$ are homeomorphic. In the fourth  {subsection} we will construct smooth atlases with which we show that manifolds $(N_j, g_j)$ are diffeomorphic. In the fifth  {subsection} we will use techniques developed in papers \cite{Ma} and \cite{MaTo} to prove that manifolds $(N_j, g_j)$ are isometric.

 {Finally, in Section \ref{subset: application} we give an example how the main result can be applied for an inverse source 
problem for a geometric wave equation.}


\subsection{Embeddings of a Riemannian manifold.}
{
A classical distance function representation of a Riemannian manifold is 
 the Kuratowski-Wojdyslawski embedding, 
 $$\mathcal K:x\mapsto \dist_M(x,\cdotp),$$
 from $M$  to the space  of continuous functions $C(M)$  on it.
 The mapping $\mathcal K:M\to C(M)$ is an isometry
so that $\mathcal K(M)$ is an isometric representation of $M$ in a vector space.

Another important example is the Berard-Besson-Gallot representation \cite{Besson}
\ba
& &\mathcal G:M\to C(M\times \R_+), \quad \mathcal G(x)=\Phi_M(x,\,\cdotp,\,\cdotp)
\ea
where
 $(x,y,t)\mapsto \Phi_M(x,y,t)$  is the heat kernel of the manifold $(M,g)$.
 The asymptotics of the heat kernel $\Phi_M(x,y,t)$,  as $t\to 0$,
 determines the distance $d(x,y)$, and by endowing $C(M\times \R_+)$ with
 a suitable topology, the image $\mathcal G(M)\subset C(M\times \R_+)$
 can be considered as an embedded image of the manifold $M$.

\mtext{
Theorem \ref{main theorem} implies that the set $\D(M)=\{D_x;\ x\in M\}$
can be considered as  an embedded image (or a representation) of the manifold $(M,g)$ in the space 
$C(F\times F)$ in the embedding $x\mapsto D_x$.
Moreover, in the proof of Theorem \ref{main theorem} we show that $(F,g|_F)$ and 
the set $\D(M)$ determine uniquely an atlas of differentiable coordinates and a metric tensor on
$\D(M)$. These structures   make $\D(M)$ a Riemannian manifold that is 
isometric to the original manifold $M$. Note that the metric is different than
the one inherited from the inclusion $\D(M)\subset C(F\times F)$. Hence, $\D(M)$ can be considered as a representation of the manifold $M$, given
in terms of the distance difference functions, and we call it the {\it distance difference representation} of  the manifold of $M$ in $C(F\times F)$.

The embedding $\D$ is different to the above embeddings $\mathcal K$ and $\mathcal G$
in the following way that makes it important for inverse problems: With $\D$
one does not  need to know a priori the set $M$ to consider the function space $C(F\times F)$
into which the manifold $M$  is embedded.
Similar types of embedding have been also considered in the context of boundary distance functions, see Subsection \ref{subsec: RM}.

In addition to the above tensor $g$ on $N$, let us consider a sequence of metric tensors $g_k$, $k\in \Z_+$  on the manifold $N$ and
assume that  $g_k|_F=g|_F$  on $F\subset N$.
We denote the Riemannian manifolds
$(N\setminus F,g_k|_{N\setminus F})$, having the boundary $\p F$, by
 $(M_k,g_k)$. Also,  we denote
by $\D(M_k)\subset C(F\times F)$ the distance difference representations of  the manifolds
$(M_k,g_k)$ and let  $d_H(X_1,X_2)$ denote the  Hausdorff distance of sets $X_1,X_2\subset C(F\times F)$.
When $d_H(\D(M_k),\D(M))\to 0$, as $k\to \infty$,
an interesting open question is, if the manifolds $(M_k,g_k)$ converge to $(M,g)$ in the Gromov-Hausdorff 
topology.
This type of questions have been studied for other representations e.g.\ in \cite{AKKLT,Besson}, but
this question is outside the context of this paper.}}

\subsection{Earlier results and the related inverse problems}
\label{sec: earlier}

The inverse problem for the distance difference function is closely related to 
many other inverse problems. We review some results below:

\subsubsection{Boundary distance functions and the inverse problem for a wave equation}
\label{subsec: RM}

The reconstruction of a compact Riemannian manifold $(M,g)$  with boundary from distance information has been considered 
e.g.\ in  \cite{IBSP,Ku5}. There, one defines for $x\in M$ the boundary distance function
$r_x:\p M\to \R$ given by $r_x(z)=d(x,z)$.  Assume that one is given the boundary $\p M$
and the collection of boundary distance functions corresponding to all $x\in M$, that is, 
\begin{equation}
\p M\quad\hbox{and}\quad \mathcal{R}(M):=\{r_x\in C(\partial M);\ x\in M\}.
\label{boundary distance data}
\end{equation}
It is shown in \cite{IBSP,Ku5} that only knowing the boundary distance data \eqref{boundary distance data}
one can reconstruct the topology of $M$, the differentiable structure  of $M$ (i.e., an atlas of
$C^\infty$-smooth coordinates), and the Riemannian metric tensor $g$.
\mtext{Thus $\mathcal{R}(M)\subset C(\p M)$ can be considered as an isometric copy of $M$,
and the pair $(\p M,\mathcal{R}(M))$  is called the boundary distance representation of $M$,
see \cite{IBSP,Ku5}.}
 \matti{Similar results for non-compact manifolds   are considered in
\cite{IKL}. Constructive solutions to determine the metric from the boundary distance functions
have been developed in \cite{deHoop1, deHoop2} using a Riccati equation {\cite{Pe} for metric tensor
in boundary normal coordinates} and in \cite{Pestov-Uhlmann} using the properties of the conformal 
killing tensor.  


Physically speaking, functions $r_x$ are determined
by the wave fronts of waves produced by the delta-sources $\delta_{x,0}$ that take place at the point $x$ at time $s=0$. The distance difference functions $D_x^{\p M}$ are determined by
the wave fronts of waves produced by the delta-sources $\delta_{x,s}$ that take place at the point $x$ at an unknown time $s\in \R$. 

Many hyperbolic inverse problems with time-independent metric reduce to 
the problem of reconstructing the isometry type of the manifold from its boundary distance functions.
Indeed, in \cite{KaKu2,IBSP,KrKL,KL2,KL3,Oksanen3,Oksanen1,Oksanen2} it has been show that the boundary measurements for the scalar wave equation,
Dirac equation, and for Maxwell's system (with isotropic scalar impedance) determine the boundary distance functions of the Riemannian
metric associated to the wave velocity.

} 


\subsubsection{Hybrid inverse problems}
\matti{Hybrid inverse problems are based on
coupling two physical models together. In a typical setting of these problems, 
the first physical system is such that by
controlling the boundary values of its solution, one can produce high amplitude waves, 
that create, e.g. due to energy absorption, a source for the second physical system.
Typically, the second physical system corresponds to a hyperbolic equation with {the metric
$$ds^2=c(x)^{-2}((dx^1)^2+\dots+(dx^n)^2)$$  corresponding to}
 the wave speed $c(x)$. Examples of such hybrid inverse problems
 are encountered in thermo-acoustic and photo-acoustic imaging
 see e.g.\ \cite{Ammari1,Bal1,Bal2,Bal3,Bal4,StU1,StU1B,StU3,StU4}
 and quantitative elastography \cite{Bercoff,Hoskins,Jeong}.
 In some cases one can use beam forming in the first physical system
 to make the source for the second physical system to be strongly localized,
 that is, to be close to a point-source, see e.g. \cite{Bercoff,Jeong}.
   
   To simplify the above hybrid inverse problem, one often can do approximations
 by assuming that the wave speed in the second physical system is either a constant
or precisely known. Usually one also assumes that the time moment when the
source for the second physical system is produced is exactly known.  However, when these approximations are not made,
the wave speed $c(x)$ needs to be determined, too. {When
the source of the second physical system is produced at the given time
in the whole domain $M$, the problem is studied in \cite{Liu,StU2}.
In the cases when 
the source of the second physical system} 
are close to a point sources,  one can try to determine
$c(x)$ from the wavefronts that are produced by the point sources and are observed outside the domain $M$. 
This problem can be uniquely solved by
Theorem \ref{main theorem} and we consider it in detail below in
Section \ref{subset: application}.

}

\subsubsection{Inverse problems of  micro-earthquakes.}
The earthquakes are produced by
 the accumulated elastic strain that at some time suddenly  produce an earthquake. As mentioned above, the small magnitude earthquakes (e.g. the micro-earthquakes of magnitude $1<M<3$) appear so frequently 
 that the surface observations of the produced elastic waves have been proposed to be used in the imaging of the 
Earth near  active faults \cite{Kayal,Sava}.  {The so-called time-reversal techniques to study
the inverse source and medium problems arising from  the micro-seismology have been developed in
\cite{quake1,dHT,quake2}.}

 In geophysical studies, one often 
approximates the elastic waves with scalar waves satisfying a wave equation. Let us also
assume that the sources of such earthquakes are point-like and that one does measurements over so long time that 
the source-points are sufficiently dense  in the studied volume. Then the inverse problem 
 of determining the 
the speed of the waves in the studied volume  from the surface observations of the micro-earthquakes is close
to the problem studied in this paper. We note that the above assumptions are highly idealized: For example, considering
the system of elastic equations would lead to a problem where travel times are determined by a Finsler metric instead
of a Riemannian one. One possible way to continue the line of research conducted in this paper, would be to study, if the result of Theorem \ref{main theorem} holds on Finsler manifolds. The authors have not yet addressed this issue.


\subsubsection{Broken scattering relation}
\matti{If the sign in the definition of the distance difference functions is changed 
in (\ref{def: dist diff}), we come to  distance sum functions
\beq\label{def: dist sum}
D^+_x(z_1,z_2)=d(z_1,x)+d(z_2,x),\quad x\in M,\ z_1,z_2\in N\setminus M.
\eeq
This function gives the length of the broken geodesic
that is the union of the shortest geodesics connecting $z_1$ to $x$ and
the shortest geodesics connecting $x$ to $z_2$. Also, the gradients of $D^+_x(z_1,z_2)$ with respect to 
$z_1$ and $z_2$ give the velocity vectors of these geodesics. 
The functions (\ref{def: dist sum}) appear
in the study of the radiative transfer equation on manifold $(N,g)$, see \cite{CS2,McDow1,McDow2,McDow3,Patrolia}. 
Also, the inverse problem of determining the manifold $(M,g)$ from
the broken geodesic data, consisting \mtext{of the initial and
the final points and directions, 
and the total length, of the broken geodesics, has been considered in
\cite{KLU-ajm}.}}

\section{Proof of the main result}\label{sec: proof of main result}
\subsubsection{Notations and basic facts on pre-geodesics}
When we are concerning only one manifold, we use the shorthand notations $M,N,F,$ and $g$ instead of ones with sub-indexes. 

Let $(N,g)$ be a compact and connected Riemannian $n$-manifold without boundary and $n\geq 2$. We assume that $M \subset N$ is an open set of $N$ and the set $F= N \setminus M $ is a compact,
 $F$ contains an open set
 and has a smooth boundary. Suppose  we know the Riemannian structure of manifold $(F,g|_F)$.  
\medskip

We denote the Riemannian connection of the metric $g$ by $\nabla$. A unit speed geodesic of $(N,g)$ emanating from a point \matti{$(p,\xi)\in SN$ is denoted by $\gamma_{p,\xi}(t)=\exp_p(t\xi)$. 
Here, $SN=\{(p,\xi)\in TN;\ \|\xi\|_g=1\}$.} 
We use a short hand notation $D_t:=\nabla_{\dot{\gamma}_{p,\xi}(t)}$ for the covariant differentiation in the direction $\dot{\gamma}_{p,\xi}$ for vector fields along geodesic $\gamma_{p,\xi}$.

Let $p \in N$ and choose some smooth coordinates $(U,X)$ at point $p$. Denote the Christoffel symbols of connection $\nabla$ by $\Gamma^k_{i,j}$.

{We say that a curve $\alpha([t_1,t_2])$ is distance minimizing if the length of this curve
is equal to the distance between its end points $\alpha(t_1)$ and $\alpha(t_2)$. Also, 
a geodesic that is distance minimizing is called a minimizing geodesic.}

We say that a curve $\alpha([t_1,t_2])$ is a \textit{pre-geodesic}, if  $\alpha(t)$ is a $C^1$-smooth curve such that $\dot\a(t)\not =0$ on $t\in [t_1,t_2]$, and $\alpha([t_1,t_2])$ can be re-parameterized so that it becomes a geodesic. 

\mtext{Let us next recall some properties of the pre-geodesics. Let us consider a geodesic 
curve $\gamma:\R \rightarrow N$,  satisfying in local coordinates the equation}
\begin{equation}
D_t\dot{\gamma}(t)=\frac{d^2\gamma^k}{dt^2}(t) +\Gamma^k_{i,j}(\gamma(t))\frac{d\gamma^{i}}{dt}(t)\frac{d\gamma^{j}}{dt}(t)=0, \: k\in \{1,\ldots,n\}.
\label{geodesic equation 1}
\end{equation}
We need the following result, often credited to 
 Levi-Civita \cite{LC}.
\begin{Le}
Let $\kappa:\R\rightarrow \R$ be continuous and $\tilde \gamma:\R \rightarrow N$ be a $C^2$-curve that satisfies the  equation
\begin{equation}
\frac{d^2\tilde \gamma^k}{ds^2}(s) +\Gamma^k_{i,j}(\tilde \gamma(s))\frac{d\tilde \gamma^i}{ds}(s)\frac{d\tilde \gamma^j}{ds}(s)=\kappa(s)\frac{d\tilde \gamma^k}{ds}(s),  \: k\in \{1,\ldots,n\}.
\label{geodesic equation 4}
\end{equation}
Then there exists a change of parameters $t:\R \rightarrow \R$ satisfying
\begin{equation}
\frac{dt}{ds}(s)
=\exp\bigg(\int\limits_0^{s}\kappa(\tau)d\tau\bigg).
\label{proper change of variables for geo eq}
\end{equation}
such that curve $\gamma(t):=\tilde \gamma(s(t))$ solves the geodesic equation (\ref{geodesic equation 1}). Here $s(t)$ is the inverse function for $t(s)$.
\label{Equivalence of geodesic formulas}
\end{Le}
\begin{proof} \mtext{The proof is a direct computation.}
\end{proof}
Let us now consider a family $\mathcal{C}$ of  $C^2$-smooth curves defined on $U$. We denote by  $\Omega$ the subbundle of $TU$ that is deternined by  the velocity fields $(c,\dot c)$,  $c \in \mathcal{C}$.   More precisely a vector $(p,v) \in TU$ satisfies $
(p,v) \in \Omega \hbox{ if and only if there exist } a, t\in \R, c \in \mathcal{C} \hbox{ such that }  (p,v)=(c(t),a\dot{c}(t)).$ 
Let $f: \Omega \rightarrow \R$ be a function that satisfies 
\begin{equation}
f(av)=af(v), \textrm{ for all }a\in \R \textrm{ and } v \in\Omega,
\label{homogeneous of degree 1}
\end{equation}
i.e., it is homogeneous of degree 1. Moreover we assume that $f$ satisfies the equation 
\begin{equation}
\frac{d^2\tilde \gamma^k}{ds^2}(s) +\Gamma^k_{i,j}(\tilde \gamma(s))\frac{d\tilde \gamma^i}{ds}(s)
\frac{d\tilde \gamma^j}{ds}(s)= f\bigg(\frac{d\tilde \gamma}{ds}(s)\bigg)\frac{d\tilde \gamma^k}{ds}(s),
\label{geodesic equation 3}
\end{equation}
for any $\widetilde \gamma \in \mathcal{C}$. By Lemma \ref{Equivalence of geodesic formulas} each $\widetilde{\gamma} \in \mathcal{C}$ is a pre-geodesic of connection $\nabla$. 

Next we will show that also the  converse result for the pre-geodesics hold. Let $\widetilde \gamma$ a pre-geodesic passing over the point $p$. We assume that $\widetilde \gamma(0)=p$. Let $s(t)$ be such a re-parametrization of $\widetilde \gamma$ that $\widetilde \gamma(s(t))=: \gamma(t)$ satisfies the geodesic equation \eqref{geodesic equation 1}, $s(0)=0$ and $\frac{d}{dt}\widetilde \gamma(s(t))|_{t=0}\in S_pN$. Then by the chain rule it holds that 
$$
\frac{d^2\widetilde\gamma^k}{ds^2}(s) +\Gamma^k_{i,j}( \widetilde\gamma(s))\frac{d \widetilde\gamma^i}{ds}(s)\frac{d \widetilde \gamma^j}{dt}(s)=-\frac{\ddot s(t)}{\dot{s}(t)^2}\frac{d \widetilde\gamma^k}{ds}(s),  \: k\in \{1,\ldots,n\}.
$$
Let $\Omega$ be the subbundle of $TU$ that is determined by the velocity vectorfield $(\gamma, \dot{\gamma})$. We define $f:\Omega \to \R$ 
$$
f(q,v)=
\mp \|v\|_g\frac{\ddot s (t)}{\dot{s}(t)^2}, \hbox{ if } \frac{v}{\|v\|_g}=\pm\dot{\gamma}(t), \hbox{ for some } t \in \R .
$$

Thus equations \eqref{geodesic equation 1} and \eqref{geodesic equation 3} are equivalent in the sense that curves satisfying the latter one, for appropriate $f$, are also geodesics of metric $g$, but para\-metrized in a different way. 
\medskip

The distance function of $N$ is denoted by {$d(x,y)=d_N(x,y)$ for $x,y\in N$.} 
The normal vector field of $\p M$, pointing inside $M$, is denoted by $\nu$. 
The boundary cut locus function is $\tau_{\partial M}: \partial M \rightarrow \R_+,$
\begin{equation}
\tau_{\partial M}(z)=\sup\{t>0;\ d(\gamma_{z,\nu}(t), \p M)=t \}.
\label{dist to boundary cut point}
\end{equation}
Also, we use the cut locus function of $N$  that is $\tau: TN \rightarrow \R_+$, 
\begin{equation}
\tau(x,\xi)=\sup\{t>0;\ d(\exp_x(t\xi),x)=t\}
\label{dist to cut point}.
\end{equation}
Functions $\tau_{\p M}(z)$ and $\tau(x,\xi)$ 
are continuous and satisfy the inequality (see {Lemma 2.13} of \cite{IBSP})
\beq\label{tau inequalities}
\tau(z,\nu(z))> \tau_{\p M}(z),\quad z\in \p M. 
\eeq
\subsection{Extension of data}

In this subsection we prove Proposition \ref{bigger data}.

Let 
 $z_1, z_2 \in \partial F=\partial M$.
Then using the triangular inequality and that $d(z_1,z_2)=D_{z_2}(z_1,z_2)$  we see easily that
\begin{equation}
d(z_1,z_2)=\sup_{x \in  M}D_x(z_1,z_2).
\label{boundary distance}
\end{equation}
Thus $\mathcal D(M)$ determines the distances of the boundary points,
that is, the function $d|_{\p M\times \p M}:\p M\times \p M\to \R$.


\begin{Le}
Suppose that \eqref{Data1a}-\eqref{Data1b} are valid. Then for every $w,z \in F_1$ it holds that $d_1(w,z)= d_2(\phi(w),\phi(z)).$
\label{distance coincide outside M}
\end{Le}

The proof of the lemma below is very simple, but as Lemma \ref{distance coincide outside M} shows how the given data is extended to a larger data set, we present its proof. Notice that Lemma \ref{distance coincide outside M} proofs $(i)$ of the Proposition \ref{bigger data}. 

\begin{proof} Let $w,z \in F_1$. Let $\gamma$ be a minimizing unit speed geodesic in $N_1$ from $z$ to $w$ and denote $S=\gamma([0,d_1(w,z)])\cap \partial M_1$. {When $S= \emptyset$, using $\phi^\ast g_2=g_1$
we see that $d_1(w,z)\geq d_2(\phi(w),\phi(z))$.}

Next, consider the case when $S\neq \emptyset$.  Let $e_1,e_2 \in S$ be such that 
$$ 
d_1(w,e_1)=\min\{d_1(w,x):x\in S\} \textrm{ and } d_1(z,e_2)=\min\{d_1(z,x):x\in S\}.
$$
As \eqref{Data1a}-\eqref{Data1b} is valid,  the formula \eqref{boundary distance} implies $
d_1(e_1,e_2)=d_2(\phi(e_1),\phi(e_2)).
$
Since $\phi:F_1 \rightarrow F_2$ satisfies $\phi^\ast g_2=g_1$, 
$$
\begin{array}{l}
d_1(w,z)=d_1(w,e_1)+d_1(e_1,e_2)+d_1(e_2,z)
\\
\geq d_2(\phi(w),\phi(e_1))+d_2(\phi(e_1),\phi(e_2))+d_2(\phi(e_2),\phi(z))
\\
\geq d_2(\phi(w),\phi(z))
\end{array}
$$
The opposite inequality follows by changing the roles of $N_1$ and $N_2$.
\end{proof}

{Let us consider the case when $x\in F_1$. Then,
Lemma \ref{distance coincide outside M} implies that for $z_1,z_2 \in F_1$
we have 
\begin{equation}
\label{eq:D1=D2}
\begin{array}{lll}
D^1_x(z_1,z_2)&=&d_1(x,z_1)-d_1(x,z_2)\\
&=&d_2(\phi(x),\phi(z_1))-d_2(\phi(x),\phi(z_2))\\
&=&D^2_{\phi(x)}(\phi(x),\phi(z_2)).
\end{array}
\end{equation}
Hence,
\beq\label{eq subset}
\{D^1_x(\cdot,\cdot)\ ;\ x\in F_1\}\subset \{D^2_y(\phi(\cdot),\phi(\cdot))\ ;\ y \in F_2\}.
\eeq
Changing roles of $N_1$ and $N_2$  and considering $\phi^{-1}:F_2\to F_1$ instead of  the diffeomorphism
$\phi:F_1\to F_2$, we see that in formula (\ref{eq subset}) we have the equality.
This and formula (\ref{Data1b}), together with
Lemma \ref{distance coincide outside M}, imply Proposition \ref{bigger data}.}
\hfill q.e.d.

\subsection{Manifolds $N_1$ and $N_2$ are homeomorphic.}
 \matti {To simplify the notations, we will next in our considerations omit the sub-indexes of sets $M_1,N_1$, and $F_1$ and
just consider the sets $M,N$, and $F$.}

Let $x\in N$ and define a function $D_x:F \times F \rightarrow \R$ by a formula 
$$
D_x(z_1,z_2)=d(x,z_1)-d(x,z_2).
$$ 
Let $\mathcal{D}:N\rightarrow C(F \times F)$ be given by $\mathcal{D}(x)=D_x$.
Next, we consider the function space $C(F \times F)$ with the norm $\|f\|_{\infty}=\sup_{x,y\in F} |f(x,y)|$
\begin{Th}
The image $\D(N)=\{D_x;\ x\in N\}\subset C(F \times F)$ is a topological manifold homeomorphic to manifold $N$. Moreover, $\mathcal{D}(M)$ is homeomorphic to $M$.
\label{topolog. recons lemma}
\end{Th}
\begin{proof} The proof consists of four short steps.
\medskip

\textit{Step 1}
First, we will show that
$\D$ is 2-Lipschitz and therefore continuous.
Let 
 $x,y\in N$. Using the triangular inequality we see that
\beq\nonumber 
\|D_x-D_y\|_{\infty}&=&\sup_{z_1,z_2\in F} |D_x(z_1,z_2)-D_y(z_1,z_2)|
\\ \label{eka}
& &\hspace{-2cm} \leq  \sup_{z_1,z_2\in F} |d(x,z_1)-d(y,z_1)|+|d(x,z_2)-d(y,z_2)|
\\ \nonumber
&& \hspace{-2cm} \leq 2d(x,y).
\eeq
\textit{Step 2.} Next we will show that $\D$ is injective. 
{Suppose that $x,y \in N$ are such that $D_x=D_y$ and $x\not =y$. 
Let $q\in F^{int}$ and denote 
$\ell_x=d(q,x)$ and $\ell_y=d(q,y)$. Next, without loss of generality, we assume that $\ell_x\leq \ell_y$.
Also, let  $\eta\in S_{q}N$ be 
such that $\gamma_{q,\eta}([0,\ell_x])$ is a minimizing geodesic from $q$ to $x$.
Let $s_1>0$  be such that $s_1<\min(\ell_x,\ell_y)$  and $\gamma_{q,\eta}([0,s_1])\subset F^{int}$.  
Consider a point $p=\gamma_{q,\eta}(s)$ with $s\in [0,s_1]$. Then we see that  
\ba
(d(q,p)+d(p,y))-d(q,y)&=&d(q,p)+D_y(p,q)\\
&=&d(q,p)+D_x(p,q)\\
&=&(d(q,p)+d(p,x))-d(q,x)=0
\ea
and hence $p$  is on a minimizing geodesic
from $q$ to $y$. 

Let us consider a minimizing geodesic $\a$ from $p$ to $y$ with the length
$\ell_y-s$.
Then the union of the geodesics $\gamma_{q,\eta}([0,s])$  and $\a$  
is a distance minimizing curve from $q$ to $y$ and thus this union is a geodesic.
This implies that $\a$ is a continuation of the geodesics $\gamma_{q,\eta}([0,s])$
and hence $y=\gamma_{q,\eta}(\ell_y)$. Summarizing,
$\gamma_{q,\eta}([0,\ell_x])$  and $\gamma_{q,\eta}([0,\ell_y])$
are distance minimizing geodesics from $q$ to $x$ and $y$, respectively.
Since $x\not =y$, we have $\ell_x\not =\ell_y$. Then,
as we have assumed that   $\ell_x\leq \ell_y$, we see that  $\ell_x< \ell_y$.

Let $\hat q\in F^{int}$ be a such point that $\hat q$ is not on 
the curve  $\gamma_{q,\eta}(\R)$. Clearly, such a point exists as
$N$  has the dimension  $n\ge 2$. 
Let $\hat \ell_x=d(\hat q,x)$ and $\hat \ell_y=d(\hat q,y)$. Also, let  $\hat \eta\in S_{\hat q}N$ be 
such that $\gamma_{\hat q,\hat \eta}([0,\hat \ell_x])$ is  minimizing geodesic from $\hat q$ to $x$.
As above, we see that then 
$\gamma_{\hat q,\hat \eta}([0,\hat \ell_x])$  and $\gamma_{\hat q,\hat \eta}([0,\hat \ell_y])$
are distance minimizing geodesics from $\hat q$ to $x$ and $y$, respectively.
However, the geodesics $\gamma_{q,\eta}(\R)$ and
$\gamma_{\hat q,\hat \eta}(\R)$ do not coincide as point sets and hence 
the vectors $\dot\gamma_{q,\eta}(\ell_x)\in T_xN$ and   $\dot\gamma_{\hat q,\hat \eta}(\hat \ell_x)\in T_xN$
are not parallel.
Recall that $\ell_x<\ell_y$.
In the case when  $\hat \ell_x<\hat \ell_y$, let $\beta$ be the geodesic segment
$\gamma_{\hat q,\hat \eta}([\hat \ell_x,\hat \ell_y])$ connecting $x$ to $y$.
In the case when  $\hat \ell_x>\hat \ell_y$, let $\beta$ be the geodesic segment
$\gamma_{\hat q,\hat \eta}([\hat \ell_y,\hat \ell_x])$ connecting $x$ to $y$.

Then we see that the union of the paths $\gamma_{q, \eta}([0,\ell_x])$  and $\beta$
is a distance minimizing path from $q$ to $y$. As the vectors $\dot\gamma_{q,\eta}(\ell_x)$ and   $\dot\gamma_{\hat q,\hat \eta}(\hat \ell_x)$ are not parallel, we see that the union of these curves
 is not a geodesic. This is contradiction
and hence there are no $x,y \in N$  such that $D_x=D_y$ and $x\not =y$. 
Thus, $\D:N\to C(F\times F)$ is an injection.}

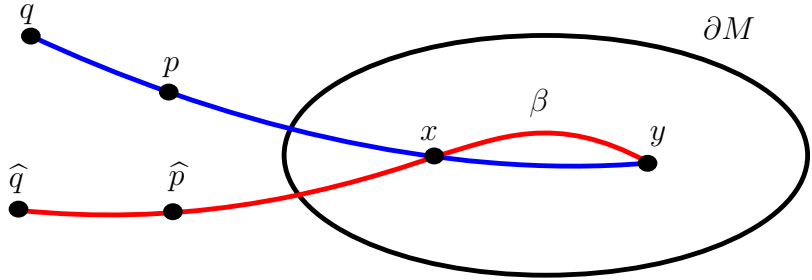
\begin{figure}[h]
 \begin{picture}(150,100)
  \put(-70,0){\includegraphics[width=300pt]{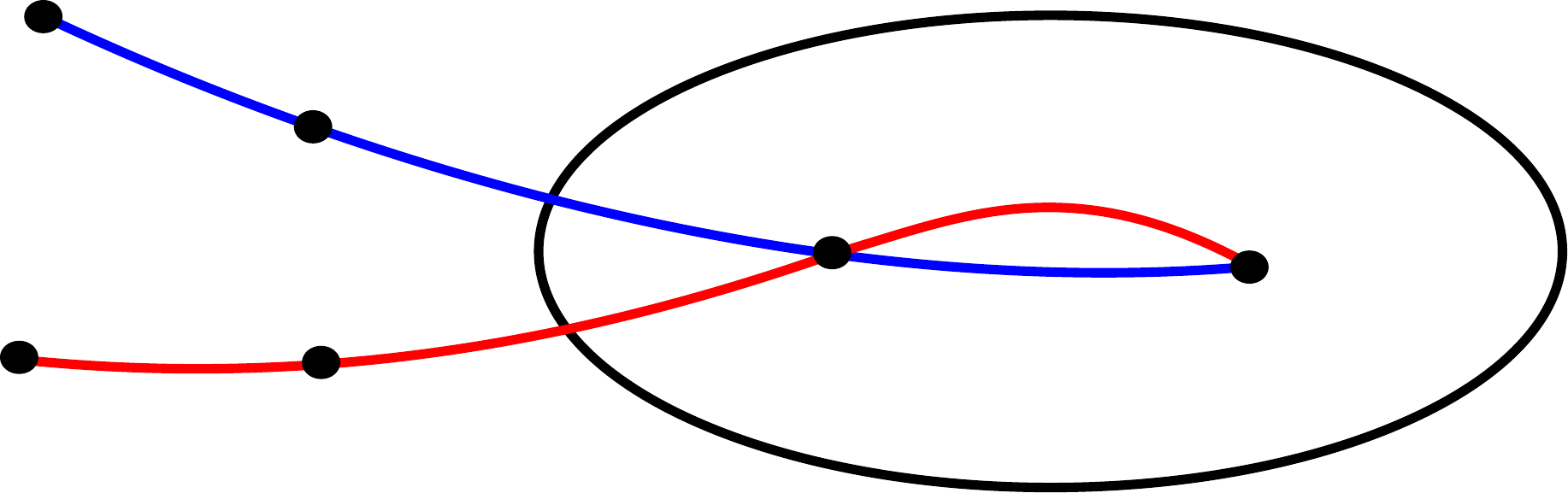}}
  \put(-66,98){$q$}
  \put(-12,78){$p$}
  \put(-70,34){$\widehat q$}
  \put(-10,35){$\widehat p$}
  \put(84,52){$x$}
  \put(170,52){$y$}
  \put(125,63){$\beta$}
  \put(190,90){$\p M$}
   \end{picture}
\caption{\it The setting in Step 2 in the proof of Theorem \ref{topolog. recons lemma}.
We consider  points $x,y\in N$  and points $p$ and $q$ such that $p$ is on a distance minimizing geodesic from
$q$ to $x$. 
Then this geodesic can be extended to a  distance minimizing geodesic from $q$ to $y$.
Similarly, the point $\hat p$ is on a distance minimizing geodesic from
$\hat q$ to $x$ and this geodesic can be extended to a  distance minimizing geodesic from $\hat q$ to $y$.
Then the union of the (blue) geodesic from $q$ to $x$ and the (red) geodesic $\beta$ is a length 
minimizing curve from $q$ to $y$ that is not a geodesic. 
}
\end{figure}

\medskip

\textit{Step 3.}
So far we have proved the continuity and injectivity of mapping $\D$. Since  the domain $N$ of the mapping $\D$ is compact and $\big(C(F \times F), \|\cdot\|_{\infty}\big)$ is a Hausdorff space as a metric space, it holds by basic results of topology that mapping $\D:N \rightarrow \D(N)$ is a homeomorphism.

\medskip

\textit{Step 4.}
By assumption $M \subset N$ is open and therefore mapping 
$\mathcal{D}:M\rightarrow \D(M)$ is open. This proves that the mapping $\mathcal{D}:M\rightarrow \D(M)$ is a homeomorphism.
\end{proof}
Define a mapping 
\begin{equation}
\Phi:C(F_2 \times F_2) \rightarrow C(F_1 \times F_1), \: \Phi(f)=f\circ (\phi \times \phi).
\label{capital Phi}
\end{equation}
Here $f\times h:X\times X \rightarrow Y\times Y$ is defined as $(f\times h)(x_1,x_2)=(f(x_1), h(x_2))\in Y \times Y$ for mappings $f,h: X \rightarrow Y$.
\begin{Th} Suppose that Riemannian manifolds $(N_1,g_1)$ and $(N_2,g_2)$ are as in section \ref{The problem setting of this paper} and {the assumptions of the Proposition \ref{bigger data} are} valid.  Let
$\Phi$ be given by (\ref{capital Phi}).
Then the mapping 
\begin{equation}
\Psi:=\D^{-1}_1\circ \Phi \circ \D_2:N_2 \rightarrow N_1
\label{capital Psi}
\end{equation}
is a homeomorphism. In addition, it holds that $\Psi^{-1}|_{F_1}=\phi$.
\label{Homeo of m1 and m2}
\end{Th}
\begin{proof} 
Due the Theorem \ref{topolog. recons lemma}, for the first claim, we only have to prove that mapping $\Phi$ is a homeomorphism. Note that mapping $\Phi$ has an inverse mapping $h\mapsto h\circ (\phi^{-1}\times \phi^{-1})$. Let $(x,y)\in F_1 \times F_1$ and $f,h\in C(F_2 \times F_2)$ then it follows
$$
|(\Phi(f)-\Phi(h))(x,y)|=|f(\phi(x),\phi(y))-h(\phi(x),\phi(y))| \leq \|f-h\|_{\infty}.
$$
This proves the continuity of $\Phi$. A similar argument where $\phi$ is replaced by $\phi^{-1}$ proves that mapping 
{$\Phi$
is a homeomorphism.}

Let $x\in F_1$ and denote $y=\phi(x)$. Then
\ba
 \Psi^{-1}(x)&=&(\D^{-1}_2\circ \Phi^{-1} \circ \D_1)(x)=\D^{-1}_2(D^1_x(\phi^{-1}(\cdot)\times \phi^{-1}(\cdot)))
\\ & \stackrel{\eqref{eq:D1=D2}}{=}&\D_2^{-1}(D^2_y)=y.
\ea
\end{proof}

\begin{remark}\label{rem: density}
\matti{As the map $\D:M\to \D(M)$, {$x\mapsto D_x$,} is a homeomorphism, we see that for a dense set
$X\subset M$ we have
\ba
\D(M)=\hbox{cl}(\D(X))=
\hbox{cl}\{D_x\ ;\ x\in X\}\subset  C(F\times F)\},
\ea
where the closure cl is taken with respect to the topology of $C(F\times F)$.
This means that the distance difference functions corresponding to $x$ in a dense set $X$
determine the distance difference functions corresponding to the points in the whole set $M$.}
\end{remark}

\subsection{Manifolds $N_1$ and $N_2$ are diffeomorphic.}
\label{smooth structure}
Our next goal is to construct such smooth atlases for manifolds $N_i$ that homeomorphism $\Psi:N_2 \rightarrow N_1$ of Theorem \ref{Homeo of m1 and m2} is a diffeomorphism. 

%

\begin{Le}
\label{Le:smooth_coord}
Let $(N,g)$ be a compact Riemannian manifold of dimension $n$, $x\in N$ and $\xi \in T_x N$, $\|\xi\|_g=1$. Let $\gamma_{x,\xi}:[0,\ell]\rightarrow N$ be a distance minimizing geodesic. Let $0<h<\ell$,  $z=\gamma_{x,\xi}(h)$. Then there exists a neighborhood $V$ of $z$ such that the set  
\begin{equation}
\label{eq:def_U}
U=\{(z_i)_{i=1}^n \in V^n: \dim \hbox{span}((F(z_i)-\xi)_{i=1}^n)=n\} 
\end{equation}
is open and dense in $V^n:= V\times V \times \ldots \times V$. Here $F(q):=\frac{(\exp_x)^{-1}(q)}{\|(\exp_x)^{-1}(q)\|_g}, \: q \in V$. 

Moreover for every $(z_i)_{i=1}^n \in U$ there exists an open neighborhood $W$ of $x$ such that 
\begin{equation}
\label{eq:coordinate_map}
H:W\rightarrow \R^n, \: H(y)=(d(y,z_i)-d(y,z))_{i=1}^n
\end{equation}
is a smooth coordinate mapping. 
\label{Smooth coordiante system}
\begin{figure}[h]
 \begin{picture}(150,180)
  \put(-60,-10){\includegraphics[width=300pt]{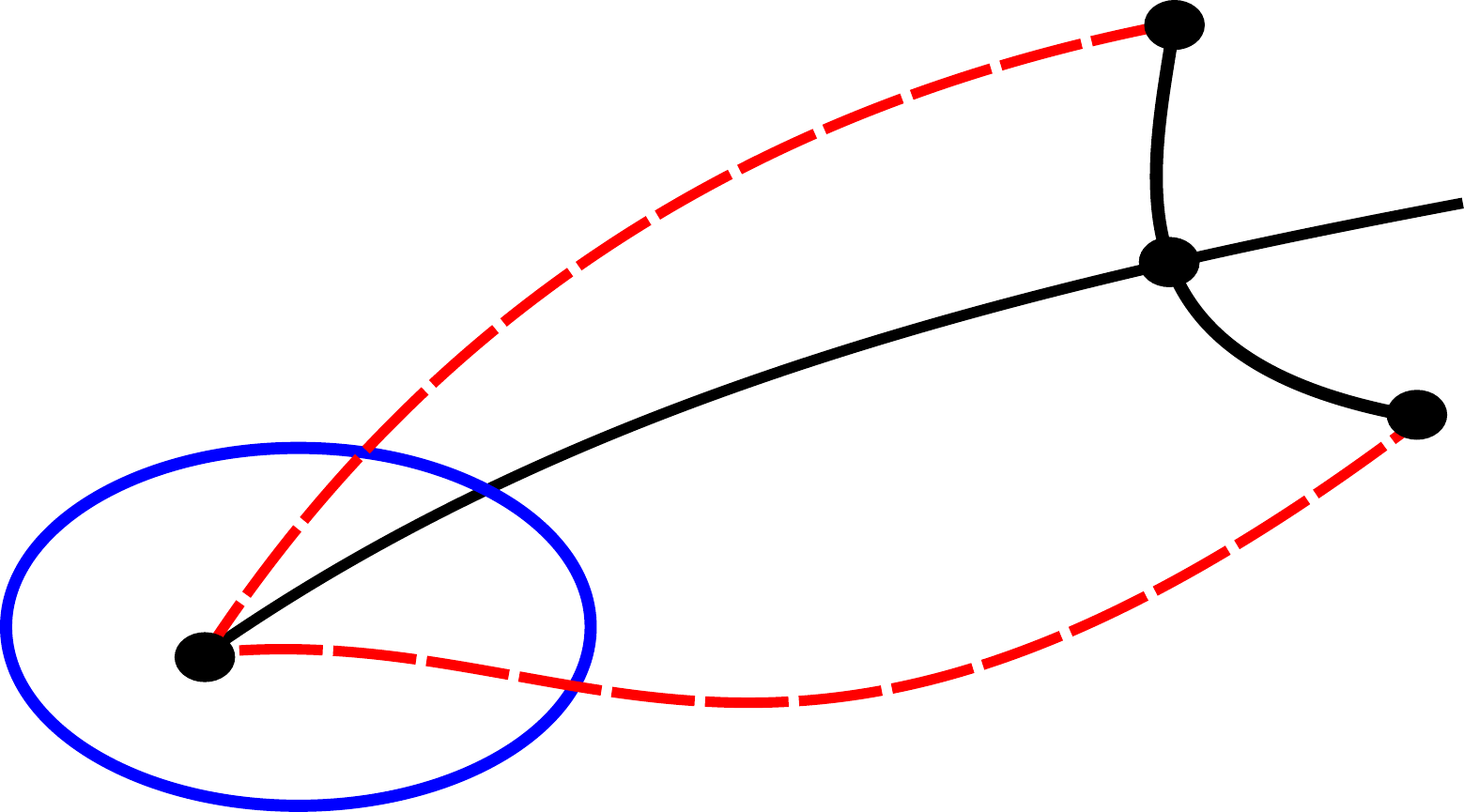}}
  \put(-25,30){$x$}
  \put(135,80){$\gamma_{x,\xi}$}
  \put(185,110){$z$}
  \put(195,150){$z_1$}
  \put(240,75){$z_2$}
  \put(20,5){$W$}
   \end{picture}
   \caption{\it A schematic picture of the coordinate system $H$.}
\end{figure}
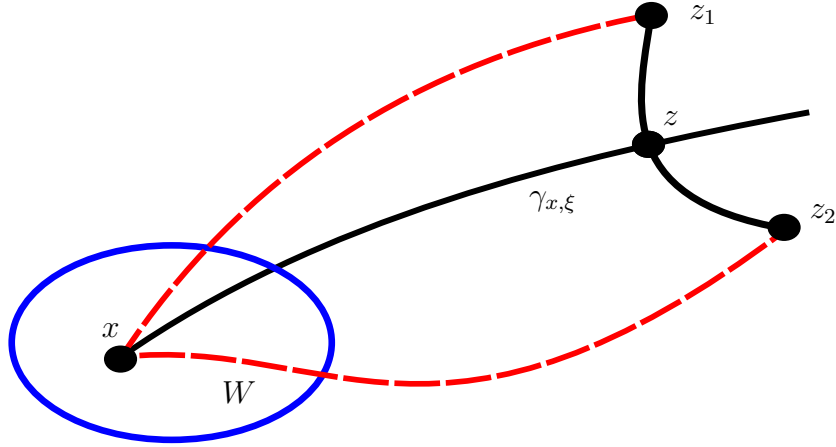
\end{Le}
\begin{proof}

Since the geodesic $\gamma_{x,\xi}([0,\ell])$ is distance minimizing, the geodesic segment $\gamma_{x,\xi}([0,h])$ from $x$ to $z$ has no cut points.
Moreover,  there exist neighborhoods $V_x$ and $V$ of $x$ and $z$ such that the mapping $(p,q)\mapsto d(p,q)$ is smooth on $V_x\times V$.
As the geodesic $\gamma_{x,\xi}([0,h])$  has no cut points, 
the differential of $\exp_x$ at $v:=h\xi \in T_xN$ is invertible. In particularly the map $F:V \to S_xN$ is well defined and smooth. 

Now we study the properties of the set $U$, given in \eqref{eq:def_U}. Consider the function 
$$
T:(S_xN)^n \to \R, \: T((v_i)_{i=1}^n)=\det(v_1-\xi, \ldots, v_n-\xi). 
$$
Then it holds that $(z_i)_{i=1}^n \in U$ if and only if $T((F(z_i))_{i=1}^n)\neq 0$. Therefore the set $U$ is open.  

We define a set 
$$
O:=T^{-1}(\R \setminus \{0\}) \subset (S_xN)^n.
$$ 
Then $O$ is open. Our aim is to prove that the set $O$ is also dense.  We note that $(S_xN)^n$ is a real analytic manifold and the map $T$ is real analytic since, it is a polynomial. Also the constant map $0=:(v_i)_{i=1}^n\mapsto 0$ is real analytic. By Lemma 4.3 of \cite{helga} the functions $T$ and $0$ coincide in $(S_xN)^n$ if and only if they coincide in some open subset of $(S_xN)^n$. Thus to prove that $O$ is dense, it suffices to prove that there exists $(v_i)_{i=1}^n \in (S_xN)^n$ such that $T((v_i)_{i=1}^n \neq 0$.

To simplify the notations we assume $S_xN = S^{n-1} \subset \R^n$ and $\xi=e_n$, where $e_1, \ldots, e_n$ is the standard orthonormal basis of $\R^n$. Denote $v_i =e_{i}, i \in \{1,\ldots ,n-1\}$ and $v_n = \frac{v_1+v_2}{\sqrt{2}}$. Then $T((v_i)_{i=1}^n)\neq 0$, since
$$
\hbox{span}(v_1-\xi, \ldots, v_{n-1}-\xi,\frac{v_1+v_2}{\sqrt{2}}-\xi)=\hbox{span}(e_1, \ldots, e_{n-1},e_n)=\R^n.
$$
We conclude that the set $U$ is dense in $V^n$, since $O \subset (S_xN)^n$ is dense and $F$ is an open map.

\medskip

Finally we will show that the mapping $H$, defined in \eqref{eq:coordinate_map}, is a smooth coordinate map at some neighborhood $W$ of $x$. Choose $(z_j)_{j=1}^n \in U$. By the preparations made above, it holds that the gradients 
$$
-\nabla \big(d(\cdot,z_i)-d(\cdot, z)\big)\big|_{x}=F(z_i)-\xi
$$ 
are linearly independent. Then due to the Inverse function theorem,  there exists such a neighborhood $W$ of $x$ that the function 
$$
H:W\rightarrow \R^n, \: H(y)=(d(y,z_i)-d(y,z))_{i=1}^n
$$
is a smooth coordinate mapping. 
\end{proof}
Next we consider the homeomorphism $\Psi:N_2 \rightarrow N_1$ of Theorem \ref{Homeo of m1 and m2}.

\begin{Th}
Suppose that Riemannian manifolds $(N_1,g_1)$ and $(N_2,g_2)$ are as in section \ref{The problem setting of this paper} and Proposition \ref{bigger data} is valid. Then mapping $\Psi:N_2 \rightarrow N_1$, given in formula \eqref{capital Psi},  is a diffeomorphism. 
\label{Diffeo of m1 and m2}
\end{Th} 
\begin{proof}
Note that for any $p\in N_2$ and all $q,r \in F_2$ it holds that
\begin{equation}
\label{eq:DDD_maps_and_Psi}
D_p^2(q,r)=D_{\Psi(p)}^1(\Psi(q),\Psi(r)).
\end{equation}
Let $x\in N_2, \: y \in F_2^{int}$ and denote $\widetilde{x}=\Psi(x)$ and $\widetilde{y}=\Psi(y)$. Let $h\in (0,d_2(x,y))$ be such that $z:=\gamma_{x,\xi_2}(h)\in F_2^{int}$ 
and $\gamma_{x,\xi_2}([h,d_2(x,y)])\subset F^{int}_2$, where $\gamma_{x,\xi_2}$ is a minimizing unit speed geodesic from $x$ to $y$ and $\widetilde{z}=\Psi(z )\in F_1^{int}$. Note that by the choice of $z$ it holds that it is not a cut point of $x$ on curve $\gamma_{x,\xi_2}$. Therefore mapping $p\mapsto D_{p}^2(r,q)$ is $C^\infty$-smooth, when $p$ is sufficiently close to $x$ and $r,q$ are sufficiently  close to $z$. 
Since 
$$
D^2_x(y,z)=D^1_{\widetilde{x}}(\widetilde{y},\widetilde{z}), \: d_2(z,y)\geq d_1(\widetilde z, \widetilde y) \hbox{ and } d_2(x,y)=d_2(x,z)+d_2(z,y),
$$
we deduce using the triangle inequality that
$$
d_1(\widetilde{x},\widetilde{y})=d_1(\widetilde{x},\widetilde{z})+{d_1}(\widetilde{z},\widetilde{y}).
$$ 
Therefore, 
there exists a unit speed distance minimizing geodesic $\gamma_{\widetilde{x},\xi_1}$ from $\widetilde{x}$ to $\widetilde{y}$ that contains the point $\widetilde{z}$. Hence, the mapping $\widetilde{p}\mapsto D_{\widetilde{p}}^1(\widetilde{r},\widetilde{q})$ is smooth, 
when
$\widetilde{p}$ is  sufficiently close to $\widetilde{x}$ and $\widetilde{r},\widetilde{q}$ are  sufficiently close to $\widetilde{z}$. 

%

Choose a neighborhood $V_2$ of $z$ such that the map $F_2:V_2 \to S_xN_2, F_2(q):=\frac{(\exp_x)^{-1}(q)}{\|(\exp_x)^{-1}(q)\|_g}, q \in V_2$ is well defined. Since $\Psi$ is homeomorphism we may assume that $\Psi(V_2)=V_1$, which  is a neighborhood of $\widetilde z$ such that the map $F_1:V_1 \to S_xN_1, F_1(q):=\frac{(\exp_x)^{-1}(q)}{\|(\exp_x)^{-1}(q)\|_g}, q \in V_1$ is well defined. 

We want to show that there exist points $(z_i)_{i=1}^n \in V_2$ for which the collections 
$$
((F_2(z_i)-\xi_2))_{i=1}^n \in T_xN_2 \hbox{ and } ((F_1(\Psi(z_i))-\xi_1))_{i=1}^n \in T_{\widetilde x}N_1 
$$
are linearly independent. Let us define
$$
U_i:=\{(z_j)_{j=1}^n \in V_i^n: \dim \hbox{span}((F_i(z_j)-\xi_i)_{j=1}^n)=n\},\: i \in \{1,2\}. 
$$
By Lemma \ref{Le:smooth_coord} the sets $U_i$ are open and dense. Since $\Psi:N_2 \to N_1$ is a homeomorphism, also the map $\Psi^n: N_2^n \to N_1^n$ defined by 
$$
\Psi^n((q_i)_{i=1}^n)=(\Psi(q_i))_{i=1}^n
$$
is a homeomorphism. Therefore $U_1 \cap \Psi^n(U_2)$ is open and dense in $V_1^n$. Due to the choice of vector $\xi_1 \in S_{\widetilde x}N_1$, there exist points $(z_i)_{i=i}^n \in U_2$ that satisfy $(\Psi(z_1))_{i=1}^n \in U_1$. 

By Lemma \ref{Le:smooth_coord} there exists a neighborhood $W_2$ of $z$ such that the map
$$
H:W_2\rightarrow \R^n, \: H(y)=(d_2(y,z_i)-d_2(y,z))_{i=1}^n
$$
is a smooth coordinate map, $W_1:=\Psi(W_2)$ is a neighborhood of $\widetilde x$ and moreover the map
$$
\widetilde H:W_1\rightarrow \R^n, \: \widetilde H(y)=(d_1(y,\Psi(z_i))-d_2(y,\widetilde z))_{i=1}^n
$$
is also a smooth coordinate map. We conclude that by equation \eqref{eq:DDD_maps_and_Psi} we have shown $H(W_2)=\widetilde H(W_1)$ and more importantly
$$
\widetilde{H}\circ \Psi \circ H^{-1}=Id
$$
Since the point $x\in N_2$ above is arbitrary and  $H$ and $\widetilde{H}$ are smooth coordinate mappings for $x$ and $\widetilde{x}$, respectively, the above implies that $\Psi$ is a diffeomorphism.
\end{proof}
\subsection{Riemannian metrics $g_1$ and $\Psi_{\ast}g_2$ coincide in $N_1$.}
In this section we will show that manifolds $(N_1,g_1)$ and $(N_2,g_2)$ that satisfy \eqref{Data1a}-\eqref{Data1b} are isometric. 
\begin{Df} Let  $z_1\in F$ and $\xi \in S_{z_1}N$. Define a set $\altomega(z_1, \xi)$ by
\beq\nonumber
\altomega_N(z_1, \xi):=\{x\in N&;& D_x(\cdot,z_1) \textrm{ is  $C^1$-smooth in a neighborhood }
\\ & & {\textrm{of $z_1$ and }} \nabla D_x(\cdot,z_1)|_{z_1}=\xi \}.
\label{Qeodesic as a point set 1}
\eeq
\end{Df}

\begin{Le}\label{altomega lemma}
Let  $z_1\in F$ and $\xi \in S_{z_1}N$.
Then it follows
\beq
\altomega_N(z_1,\xi)=\gamma_{z_1,-\xi}(\{s\ ; \ 0<s< \tau(z_1,-\xi)\}),
\label{Qeodesic as a point set 2}
\eeq
\label{Images of geodesics}
\end{Le}
Roughly speaking, Lemma \ref{altomega lemma} means that sets $\altomega_N(z_1,\xi)$, that can be determined using data \eqref{Data}, are unparameterized geodesics on $N$.

\begin{proof} First we recall that for all $x\in N$ the distance function $d(\cdot,x)$ is not smooth near $y\in N\setminus \{x\}$ if and only if point $y$ is in a cut locus of $x$. See for instance Section 9 of Chapter 5 of \cite{Pe}.


First, consider the case when $x\in \altomega_N(z_1, \xi)$. Then by the definition of $\altomega_N(z_1,\xi)$ the distance function $d(\cdot,x)$ is $C^\infty$-smooth in a neighbourhood $z_1$ so that $z_1$ is not in a cut locus of $x$, or equivalently,  $x$ is not in a cut locus of $z_1$. 
Also, have that $x\not =z_1$. Hence, there exists an unique distance minimizing unit speed geodesic from $x$ to $z_1$. Since this geodesic has the velocity 
{$$\nabla d(\cdot,x)|_{z_1}=\nabla D_x(\cdot,z_2)|_{z_1}=\xi$$} at $z_1$, it follows that $x\in \gamma_{z_1,-\xi}(\{s\ ; \ 0< s< \tau(z_1,-\xi)\})$.

Second, consider the case when $x \in \gamma_{z_1,-\xi}(\{s\ ; \ 0<s< \tau(z_1,-\xi)\})$. Then function $D_x(\cdot, z_1)$ is smooth near $z_1$ and 
$$\nabla D_x(\cdot,z_1)|_{z_1}=\dot{\gamma}(d(x,z_1))=-\dot{\gamma}_{z_1,-\xi}(0)= \xi.$$
Thus,  $x\in \altomega_N(z_1, \xi)$.
\end{proof}
The Lemma \ref{Images of geodesics} will be the key element to prove that the mapping $\Psi$ is an isometry.
\begin{Df}
Let $N$ be a smooth manifold and let   $g$ and $\widetilde{g}$ be 
metric tensors on $N$. We say that metric tensors $g$ and $\widetilde{g}$ are geodesically equivalent, if for all geodesics $\gamma:I_1 \rightarrow N $ of metric $g$ and $\widetilde{\gamma}:\widetilde{I}_1 \rightarrow N$ of metric $\widetilde{g}$ there exist changes of parameters $\alpha:I_2 \rightarrow I_1$ and $\widetilde{\alpha}: \widetilde{I_2} \rightarrow \widetilde{I_1}$ such that
$$
\gamma \circ \alpha \textrm{ is a geodesic of metric } \widetilde{g}
$$ 
and
$$
\widetilde{\gamma} \circ \widetilde{\alpha} \textrm{ is a geodesic of metric } g.
$$
\end{Df}
A trivial example of two geodesically equivalent Riemannian metrics are $g$ and $cg$, where $c>0$ is a constant. Other,
 more interesting examples of  geodesically equivalent Riemannian metrics  are 
\begin{enumerate}
\item The Southern hemisphere of the  sphere $S^2$ and the plane $\R^2$ and that are
mapped to each other in a gnomonic projection, i.e. the great circles are mapped to straight line.
\item Unit disc in $\R^2$ and the Beltrami-Klein model of a hyperbolic plane.
\end{enumerate}

Our first goal is to show that when the distance difference data on $N_1$ and $N_2$ satisfy \eqref{Data1a}-\eqref{Data1b}, 
we have  
that metric tensors $g_1$ and $(\Psi^{-1})^{\ast}g_2$ are geodesically equivalent. 
By Lemma \ref{Images of geodesics} we know all the geodesics of $N_1$ that exit unknown region $M_1$, as point sets. Next we will show that this information is enough to deduce the geodesic equivalence of $g_1$ and $(\Psi^{-1})^{\ast}g_2$.

Since the mapping $\Psi$ is a diffeomorphism, it holds that each geodesic of $(N_2,g_2)$ is mapped to some smooth curve of $(N_1,g_1)$. By {formula} \eqref{Data} and Lemma \ref{Images of geodesics}, it holds that sets $\altomega_N(z,\xi)\subset N_1$, with $z\in F_1$ and $\xi\in S_zN_1$, are  images of geodesics of $(N_2,g_2)$ in the mapping $\Psi$.
Note that the geodesic segments $\altomega_N(z,\xi)\subset N_1$   are not self-intersecting, since a cut point occurs before a geodesic stops being one-to-one. 
\medskip


Let $z\in F_2$, $\xi \in S_zN_2$ and $t_2=\tau_2(z,-\xi)$. Then  $t\mapsto \Psi(\gamma^2_{z,-\xi}(t))$, 
$t\in [0,t_2)$ is smooth and not self-intersecting curve on $N_1$. By Proposition \ref{bigger data} and Theorem \ref{Homeo of m1 and m2} we have
\begin{equation}
\label{eq:geo_segment_cut}
\Psi(\gamma^2_{z,-\xi}((0,t_2)))=\altomega_{N_1}(\Psi(z),\Psi_{\ast}\xi)=\altomega_{N_1}(\phi^{-1}(z),(\phi^{-1})_{\ast}\xi).
\end{equation}
Let $w=\phi^{-1}(z)$ and $\eta=(\phi^{-1})_{\ast}\xi$. Then by Lemma \ref{Images of geodesics} we have $\altomega_{N_1}(w,\eta)=\gamma^1_{w,-\eta}(\{s;\  0< s< t_1 \}),$ where $t_1=\tau_1(w,-\eta)$.
Furthermore, it is easy to see that there is a re-parametrization
\begin{equation}
s:(0,t_1) \rightarrow (0,t_2) \textrm{ such that } \gamma^1_{w,-\eta}(t)=\Psi(\gamma^2_{z,-\xi}(s(t))), \: t\in (0,t_1) .
\label{Psi maps geodesics to geodesics}
\end{equation}

\medskip

For $p\in N_1$, we  define a collection $\mathcal{C}(p)$ of geodesics $\gamma$ of $(N_1,g_1)$ 
and real numbers $t_0\in \R$, given by 
\ba
\nonumber 
\mathcal{C}(p)&=& \{(\gamma,t_0)\,;  \hbox{ $\gamma:(a,b)\rightarrow N_1$ is a geodesic of $(N_1,g_1)$},
\\ & &\hbox{ $\gamma(t_0)=p$, and  there are $ z\in F_1^{int},\: \xi \in S_zN_1$ }
\\& &\hbox{ such that } \gamma((a,b))=\altomega_{N_1}(z,\xi)\}.
\ea
Here $\gamma$  is given as a pair of the set $\hbox{dom}(\gamma)=(a,b)\subset \R$, $-\infty\leq a<b\leq \infty$, where
the mapping $\gamma$ is defined and the function $\gamma:\hbox{dom}(\gamma)\to N_1$. Also, $t_0\in (a,b)$.
Moreover, above $\gamma((a,b))=\altomega_{N_1}(z,\xi)$ means that the sets
$\gamma((a,b))\subset N_1$ and $\altomega_{N_1}(z,\xi)\subset N_1$ are the same, or equivalently,
that $\gamma((a,b))$ and $\altomega_{N_1}(z,\xi)$ are the same as unparameterized curves.
\medskip


%
%
For a moment we consider only metric $g_1$. 
Assume that 
$p$ is a \mtext{point in $N_1$ and $q$ is point of $F_1^{int}$ such that  $q=\gamma_{p,\xi}(\ell)$,
$\ell>0$} and the geodesic $\gamma_{p,\xi}([0,\ell])$ has no cut points. 
Then there is a neighborhood $U\subset F_1^{int}$ of $q$ and a neighborhood $V\subset T_pN_1$ of $\ell\xi$ such that $\exp_p:V\to U$ is a  diffeomorphism. Assuming that the neighborhood $V$ is small enough, we see that for any $\ell v\in V$, $\|v\|_{g_1}=1$, there is $s>0$ such that the geodesic $\gamma_{p,v}([-s,\ell])$ has no cut points. Let $s_0(p,v)\in (0,\infty]$  be the supremum of such $s$. Then, for the geodesic
$\gamma_{p,v}:(-s_0(p,v),\ell)\to N_1$ we have $(\gamma_{p,v},0)\in \mathcal{C}(p)$. 
This proves that set
\ba
\Omega_p:=\{v\in T_pN_1&;&
 \hbox{there are $(c,t_p)\in \mathcal{C}(p)$, $c(t_p)=p$ }\\
 & &\hbox{and $\dot{c}(t_p)$  is proportional to $\pm v$} \}
\ea
\mtext{contains a non-empty open double cone $\Sigma_p$, that is,
an open set that satisfies $rv\in \Sigma_p$ for all 
 $v\in \Sigma_p$ and $r\in \R\setminus \{0\}$.}
Note that the complement of $\Omega_p$ in $T_pN_1$ is non-empty,
if in manifold $M_1$ there are closed geodesics, or geodesics that are trapping in both directions in $M_1$ and go through
the point $p$.

\begin{figure}[h]
 \begin{picture}(150,150)
  \put(0,-10){\includegraphics[width=140pt]{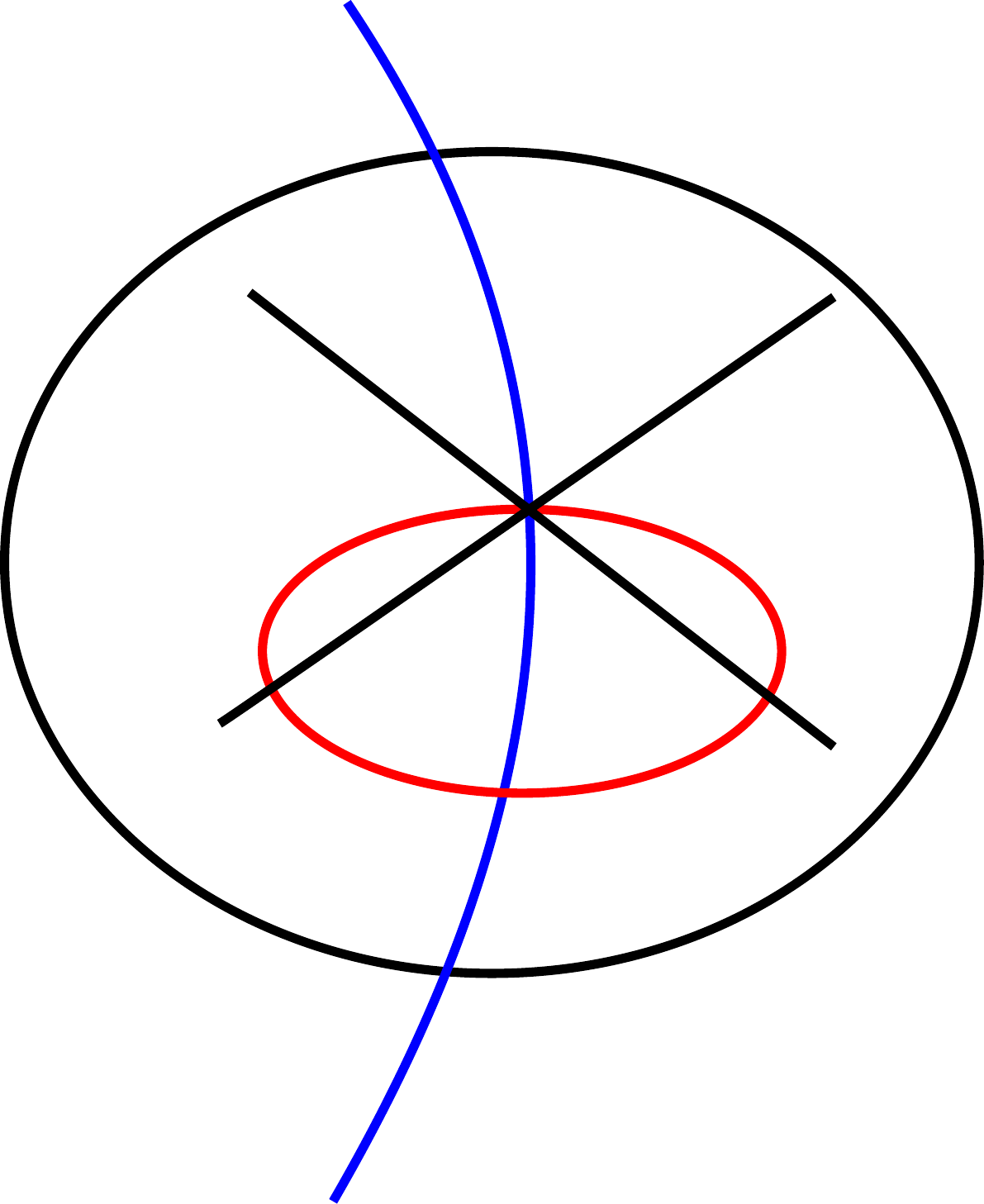}}
  \put(95,115){$\Sigma_p$}
  \put(0,30){$\p M_1$}
  \end{picture}
\caption{\it For all $p\in M_1$ there exists an open conic set $\Sigma_p\subset T_pN_1$ such that  
for every $\xi\in \Sigma_p$ the geodesic   $\gamma_{p,\xi}$ of $(N_1,g_1)$ 
can be extended to a distance minimizing geodesic (blue curve in the figure) that 
enters the set $F=N\setminus M$.  \mtext{When the distance difference data for $g_1$  and $g_2$  coincide, these geodesics have to be pre-geodesic also with respect to the metric $\Psi^{\ast}g_2$. Note that there may be  $g_1$-geodesics emanating from $p$ to directions $\xi\not \in \Sigma_p$ that do not intersect the set $F$. Such geodesics can be e.g. closed loops in $M_1$ (red curve).}}
\end{figure}


\mtext{Let point $p \in N_1$} and $(U,X)$ be coordinates near $p$,
 $X:U\to \R^n$, and denote $X(q)=(x^j(q))_{j=1}^n$. Recall 
that a pre-geodesic $\tilde \gamma$ on $(N_1,g_1)$ satisfies the formula \eqref{geodesic equation 3}, that is,
\begin{equation}
\label{two fold eq}
\bigg[\frac{d^2\tilde \gamma^k}{ds^2}(s) +\Gamma^k_{i,j}(\tilde\gamma(s))\frac{d\tilde\gamma^{i}}{ds}(s)\frac{d
\tilde\gamma^j}{ds}(s)\bigg]\bigg|_{s=s_p}=f\bigg(\frac{d\tilde\gamma}{ds}\bigg)\frac{d\tilde\gamma^k}{ds}(s)\bigg|_{s=s_p}, 
\end{equation}
$\: k\in \{1,2, \ldots ,n\}$. Here $\gamma(s_p)=p$ and $f$ is some function that is homogeneous of degree 1 on the subbundle  of $TU$ that is determined by the velocity vectorfield of $\widetilde \gamma$.

Next, we change the point of view and consider 
the  equation (\ref{two fold eq}) as a system of equations for the ``unknown'' $(\Gamma,f)$ with
the given coefficients $\frac{d\tilde \gamma}{ds}(s)|_{s=s_p}\in \Omega_p$ and $\frac{d^2\tilde \gamma}{ds^2}(s)|_{s=s_p}$ where
$(\tilde \gamma,s_p) \in \mathcal{C}(p)$.
 Here $\Gamma$ stands for a collection of Christoffel  symbols $\Gamma^k_{i,j}$ and $f:\Omega \rightarrow \R$ is a function that satisfies equation \eqref{homogeneous of degree 1} on the subbundle 
$$
\Omega :=\bigcup_{p\in U} \Omega_p \subset TU.
$$ 
 
Suppose that we also have another Riemannian connection  which Christoffel  symbols $\widetilde{\Gamma}^{k}_{i,j}$ in the $(U,X)$-coordinates have the form 
\begin{equation}
\widetilde{\Gamma}^{k}_{i,j}=\Gamma^{k}_{i,j}+\delta^k_i\varphi_j+\delta^k_j\varphi_i,
\label{Gauge freedom of connections}
\end{equation}
for some smooth functions $\varphi_i:U\to \R,$ $i=1,2,\dots,n$.
\mtext{Here, $\delta^k_i$ is one when $k=i$  and zero otherwise.}
 Let $\varphi(x)=\varphi_i(x)dx^i$ be a smooth 1-form that has  functions  $(\varphi_i)_{i=1}^n$ as the coefficients. \mtext{We need the following consequence of Lemma \ref{Equivalence of geodesic formulas}:}

\begin{Le}
Let $(U,X)$ a smooth coordinate chart. If the Christoffel symbols $\widetilde{\Gamma}$ and $\Gamma$ satisfy the equation \eqref{Gauge freedom of connections} for some 1-form $\varphi$ and pair $(f,\Gamma)$, $f: \Omega \to \R$ is homogenous of degree 1, is a solution of \eqref{geodesic equation 3}, with $s=s_p$, for 
all $(\tilde \gamma,s_p) \in \mathcal{C}(p)$, then pair $(\widetilde{\Gamma}, \widetilde{f})$ where
\begin{equation}
\widetilde{f}(v)=f(v)+2\varphi(v), \:v \in \Omega.
\label{Matveev 2nd paper tilde f}
\end{equation}
is also a solution of \eqref{geodesic equation 3}, with $s=s_p$, for all 
$(\tilde \gamma,s_p) \in \mathcal{C}(p)$.
\label{MaTo paper lemma easier direction}
\end{Le} 
\begin{proof}
Let $(\tilde \gamma,s_p) \in \mathcal{C}(p)$. A direct computation shows that 
\begin{equation}
\begin{array}{c}
(\delta^k_i\varphi_j+\delta^k_j\varphi_i)\frac{d\tilde \gamma^{i}}{ds}(s)\frac{d\tilde \gamma^{j}}{ds}(s)=\varphi_j\frac{d\tilde \gamma^{k}}{ds}(s)\frac{d\tilde \gamma^{j}}{ds}(s)+\varphi_i\frac{d\tilde \gamma^{i}}{ds}(s)\frac{d\tilde \gamma^{k}}{ds}(s)
\\
=2\frac{d\tilde \gamma^{k}}{ds}(s)\bigg(\varphi_i\frac{d\tilde \gamma^{i}}{ds}(s)\bigg)=2\frac{d\tilde \gamma^{k}}{ds}(s)\varphi\bigg(\frac{d\tilde \gamma}{ds}(s)\bigg)
\end{array}
\label{Gauge matrix}
\end{equation}
Use this and substitute equation \eqref{Gauge freedom of connections} into equation \eqref{geodesic equation 3} to obtain
$$
\frac{d^2\tilde \gamma^{k}}{ds^2}(s)+\widetilde{\Gamma}^k_{i,j}(p)\frac{d\tilde \gamma^{i}}{ds}(s)\frac{d\tilde \gamma^{j}}{ds}(s)\Bigg|_{s=s_p}\hspace{-4mm}=\hspace{-1mm}\frac{d\tilde \gamma^{k}}{ds}(s)\hspace{0mm}\bigg[f\bigg(\frac{d\tilde \gamma}{ds}(s)\bigg)+2\varphi\bigg(\frac{d\tilde \gamma}{ds}(s)\bigg)\bigg] 
$$
that proves the claim.
\end{proof}

\mtext{The following lemma gives a converse result for Lemma 
\ref{MaTo paper lemma easier direction}.
It is obtained
by using,  in a quite straightforward way,  results of V.\ Matveev \cite[Sec.\ 2]{Ma} for  general affine connections 
on pseudo-Riemannian manifolds.
However, for the convenience of the reader, 
we give   a detailed proof for the lemma and analyze at the same time the smoothness of the 1-form $x\mapsto \varphi(x)$} \mmtext{in a local coordinate neighborhood $U\subset M$.}

\begin{Le} 
Let $(U,X)$ a smooth coordinate chart. Let functions $f:\Omega\to \R $ and $\widetilde{f}:\Omega\to \R $ be homogeneous of degree 1.
Suppose that pairs  $(f,\Gamma)$ and  $(\widetilde{\Gamma}, \widetilde{f})$ both solve 
at all points $p\in U$ the system \eqref{geodesic equation 3} 
for all such coefficients $\frac{d\gamma}{ds}(s)|_{s=s_p}\in \Omega_p$ and $\frac{d^2\gamma}{ds^2}(s)|_{s=s_p}$ that $(\gamma,s_p) \in \mathcal{C}(p)$. Then the Christoffel symbols $\Gamma$ and $\widetilde{\Gamma}$ satisfy equation \eqref{Gauge freedom of connections} in $U$ with 
\mtext{a $C^\infty$-smooth}  1-form $\varphi$ in $U$.
\label{Solutions of geodesic eq are gauge equivalent}
\end{Le}

\begin{proof}
Define a pair $(\overline{f},\overline{\Gamma})$ as
$$
\overline{f}=f-\widetilde{f} \textrm{ and } \overline{\Gamma}_{i,j}^k=\Gamma_{i,j}^k-\widetilde{\Gamma}_{i,j}^k.
$$
\mtext{As a difference of two connection coefficients, $\overline\Gamma$ is a tensor.}
By substitution of pairs $(f,\Gamma)$ and  $(\widetilde{\Gamma}, \widetilde{f})$ into equation \eqref{geodesic equation 3} and by subtracting the obtained equations, we obtain at  $p\in U$ 
\begin{equation}
\overline{\Gamma}_{i,j}^kv^{i}v^j=\overline{f}(v)v^k, \textrm{ for every } v\in \Omega_p.
\label{Linear mapping chirstoffel symbols}
\end{equation}
Note that \eqref{Linear mapping chirstoffel symbols} defines a smooth extension of $\overline{f}:{\Omega}\to \R$ to $TU\setminus \{0\}$, given by 
\begin{equation}
\overline{f}(v)=\frac{\overline{f}(v)v^kg_{k\ell}v^\ell}{g(v,v)}=
\frac{\overline \Gamma^k_{i,j}(p)v^{i}v^jg_{k\ell}(p)v^\ell}{g_{ab}(p)v^av^b}, \: (p,v) \in TU\setminus \{0\}.
\label{f yläviiva jatko}
\end{equation}
Here, the rightmost term is smooth in $TU \setminus \{0\}$.
 
 Recall that $\Omega_p$ contains an open double cone $\Sigma_p\subset \Omega_p$.
\mtext{Our next goal is to show that there exist a linear function $\varphi:T_pN \rightarrow \R$ such that the restriction of function $\overline{f}$, to
 $\Sigma_p \subset \Omega_p$, 
is equal to $2\varphi|_{\Sigma_p}$}. Define a family of symmetric bi-linear mappings 
$$
\sigma^{k}:T_pN\times T_pN \rightarrow \R, \: \sigma^k(u,v)=\overline{\Gamma}_{i,j}^kv^{i}u^j, \: k \in \{1,\ldots,n\}.
$$
Since mappings $\sigma^k$ are symmetric, the parallelogram equation
$$
0=\sigma^k(u+v,u+v)+\sigma^k(u-v,u-v)-2\sigma^k(u,u)-2\sigma^k(v,v)
$$
holds.

\mtext{Next, let $u\in \Sigma_p$, $u\not =0$.} Then there is $\e=\e(u)>0$  such that, if $v\in T_pN$ satisfies $\|v\|<\e$,
then $u-v\in \Sigma_p$. 

Let us next consider $v\in \Sigma_p$  with $\|v\|<\e$. Then  $u-v,u+v \in \Sigma_p\subset \Omega_p$. By the parallelogram equality for the function $\sigma^k$ and \eqref{Linear mapping chirstoffel symbols}  we have
\begin{equation}
\begin{array}{l}
0=\overline{f}(u+v)(u+v)+\overline{f}(u-v)(u-v)-2\overline{f}(u)u-2\overline{f}(v)v
\\
\hspace{2.5mm}=\big(\overline{f}(u+v)+\overline{f}(u-v)-2\overline{f}(u)\big)u+\big(\overline{f}(u+v)-\overline{f}(u-v)-2\overline{f}(v)\big)v.
\end{array}
\label{cause of parallelogram}
\end{equation}
If vectors $u$ and $v$ are linearly independent, we get a system
\begin{equation}
\left\{\begin{array}{c}
\overline{f}(u+v)+\overline{f}(u-v)-2\overline{f}(u)=0
\\
\overline{f}(u+v)-\overline{f}(u-v)-2\overline{f}(v)=0.
\end{array}
\right.
\label{zero coefficients of linearity equation}
\end{equation}
By summing up these two equations, we get 
\begin{equation}
\overline{f}(u+v)=\overline{f}(u)+\overline{f}(v).
\label{yläviiva f}
\end{equation}
Observe that the system \eqref{zero coefficients of linearity equation} is  valid
also when $v=\lambda u$, $\lambda \in \R$.
Recall that the mappings $f$ and $\widetilde{f}$ are solutions of \eqref{geodesic equation 3} and therefore, they satisfy the equation \eqref{homogeneous of degree 1}, i.e., they commute with scalar multiplication in $\Omega_p$.

So far we have proved that $\overline{f}(u+\cdot)$ and $\overline{f}(u)+\overline{f}(\cdot)$ coincide in set $B_p(0,\epsilon)\cap \Sigma_p$. Since $\overline{f}$ is homogeneous of degree 1 it holds by \eqref{yläviiva f} that
\begin{equation}
\overline{f}(u+av)=\overline{f}(u)+a\overline{f}(v), \: v \in B_p(0,\epsilon) \cap \Sigma_p, \: -1< a< 1.
\label{f bar is linear}
\end{equation}
We define a linear function 
\beq\label{varphi def}
2\varphi:T_pN \rightarrow \R, \: 2\varphi(v)=\lim_{r \rightarrow 0}\frac{\overline{f}(u+rv)-\overline{f}(u)}{r}=\nabla_u \overline{f}(u)\cdot v.\hspace{-1cm}
\eeq
 If $v \in \Sigma_p$ and $r$ is small enough, then $rv\in B_p(0,\e)\cap \Sigma_p$ and therefore by formula \eqref{f bar is linear} it holds that 
\begin{equation}
\label{eq:phi1}
2\varphi(v)=\overline{f}(v)\quad\hbox{for every $v \in \Sigma_p$.}
\end{equation}
As $\Sigma_p$ is open, and $\varphi$ and $\overline f$ are linear,
this holds for all $v\in T_pN$ and  thus $\varphi(v)$ given by the formula 
(\ref{varphi def}) 
is independent on the
choice of used $u\in \Sigma_p$.
In local coordinates $(U,X)$ we have by \eqref{f yläviiva jatko} and \eqref{eq:phi1} that 
$$
\varphi(\frac \p {\p x^\ell}):=\frac 12 
\sum_{i,k,j=1}^n\frac 1{g_{\ell\ell}(x)}\overline \Gamma^k_{i,j}(x)\delta_\ell^{i}\delta_\ell^jg_{k\ell}(x)
$$
defines a $C^\infty$-smooth 1-form $x\mapsto \varphi(x)$  in $U$, that is an extension of $\overline{f}:\Omega \to \R$.

Define a connection
$$
\widehat{\Gamma}^k_{i,j}:= \widetilde{\Gamma}^{k}_{i,j}+\delta^k_i\varphi_j+\delta^k_j\varphi_i,
$$ 
and choose $v=\frac{d}{ds}\gamma(s)|_{s=s_p}\in \Sigma_p$. Since pairs $(f,\Gamma)$ and  $(\widetilde{\Gamma}, \widetilde{f})$ are both solutions of \eqref{geodesic equation 3} the above considerations yield
\ba
& &\bigg[\frac{d^2\gamma^{k}}{ds^2}(s)+\Gamma^k_{i,j}(p)\frac{d\gamma^{i}}{ds}(s)\frac{d\gamma^{j}}{ds}(s)\bigg]\bigg|_{s=s_p}=\bigg[f\bigg(\frac{d\gamma}{ds}(s)\bigg)\frac{d\gamma^k}{ds}(s)\bigg]\bigg|_{s=s_p}
\\
&=&\frac{d\gamma^k}{ds}(s)\bigg[2\varphi\bigg(\frac{d\gamma}{ds}(s)\bigg)+\widetilde{f}\bigg(\frac{d\gamma}{ds}(s)\bigg)\bigg]\bigg|_{s=s_p}\\
&=&\bigg[\frac{d^2\gamma^k}{ds^2}(s) +\widetilde{\Gamma}^k_{i,j}(p)\frac{d\gamma^i}{ds}(s)\frac{d\gamma^j}{ds}(s)\bigg]\bigg|_{s=s_p}+\frac{d\gamma^k}{ds}(s)\bigg[2\varphi\bigg(\frac{d\gamma}{ds}(s)\bigg)\bigg]\bigg|_{s=s_p}\\
& &\hspace{-8mm}
\stackrel{\eqref{Gauge matrix}}{=}\bigg[\frac{d^2\gamma^k}{ds^2}(s)+\widehat{\Gamma}^k_{i,j}(p)\frac{d\gamma^i}{ds}(s)\frac{d\gamma^j}{ds}(s)\bigg]\bigg|_{s=s_p}.
\ea
Therefore we have
\begin{equation}
\Gamma^k_{i,j}(p)\frac{d\gamma^{i}}{ds}(s)\frac{d\gamma^j}{ds}(s)\bigg|_{s=s_p}=\widehat{\Gamma}^k_{i,j}(p)\frac{d\gamma^{i}}{ds}(s)\frac{d\gamma^j}{ds}(s)\bigg|_{s=s_p}.
\label{Connections are same}
\end{equation}
Thus we have shown that for all $v\in \Sigma_p$ the equation 
\begin{equation}
\Gamma^k_{i,j}(p)v^{i}v^j=\widehat{\Gamma}^k_{i,j}(p)v^{i}v^j
\label{Connections are same 2}
\end{equation}
is valid. Since set $\Sigma_p$ is open, it holds that
$$
\Gamma^k_{\ell,m}(p)=\partial_{v^{\ell}v^m} \Gamma^k_{i,j}(p)v^{i}v^j=\partial_{v^{\ell}v^m}\widehat{\Gamma}^k_{i,j}(p)v^{i}v^j=\widehat{\Gamma}^k_{\ell,m}(p).
$$
\mtext{As above $p\in U$ is arbitrary,} this proves the claim.
\end{proof}
\begin{Po}
Suppose that Riemannian manifolds $(N_1,g_1)$ and $(N_2,g_2)$ are as in Section \ref{The problem setting of this paper} and \eqref{Data1a}-\eqref{Data1b} are valid. \mtext{Let $p \in N_1$} and $(U,X)$ be coordinates in a neighborhood of $p$. Then it holds that the Christoffel symbols $\Gamma$ and $\widetilde{\Gamma}$ of metrics $g_1$ and $(\Psi^{-1})^{\ast}g_2$, 
respectively, satisfy equation \eqref{Gauge freedom of connections} in $U$ with some 1-form $\varphi$, where $\Psi$ is as in \eqref{capital Psi}.
\end{Po}
\begin{proof}
Let $(U,X)$ be a local coordinate system in $N_1$. Our aim is to use the Lemma \ref{Solutions of geodesic eq are gauge equivalent} to prove the claim of this Lemma. To do so we need to  construct a function $\widetilde f:\Omega\to \R$ that satisfies $\eqref{homogeneous of degree 1}$ and moreover for any $q \in U$ the pair  $(\widetilde{\Gamma},\widetilde f)$ solves 
\mtext{the system \eqref{geodesic equation 3} 
for all such coefficients $\frac{d\gamma}{ds}(s)|_{s=s_q}\in \Omega_q$ and $\frac{d^2\gamma}{ds^2}(s)|_{s=s_q}$ that $(\gamma,s_q) \in \mathcal{C}(q)$.} 

\mtext{Let $p \in U$} and  $(c_1,t_1)\in \mathcal{C}(p)$. With out loss of generality we may assume that $t_1=0$ and $\dot{c}_1(0)=\xi \in S_pN \cap \Omega_p$.
\mtext{By definition of $\mathcal{C}(p)$, it holds that there is a unique reparametrization $t\mapsto s_\xi(t)=:s(t)$ of $c_1$ such that for curves $c_1$ and $c_2=c_1\circ s$ we have $s(0)=p$, $\dot{c}_2(0) =\xi$ and 
\begin{equation}
\label{eq:geodesic_eq_for_g1_g2}
\left\{\begin{array}{c}
\ddot{c}^k_1(t)+\dot{c}_1^{i}(t)\dot{c}_1^j(t)\Gamma^k_{i,j}(c_1(t))=0,
\\
\ddot{c}_2^k(t)+\dot{c}_2^{i}(t)\dot{c}_2^j(t)\widetilde{\Gamma}^k_{i,j}(c_2(t))=0.
\end{array}\right.
\end{equation}
Using the chain rule we can write the latter equation as
$$
\ddot{c}^k_1(s(t))+\dot{c}_1^{i}(s(t))\dot{c}_1^j(s(t))\widetilde{\Gamma}^k_{i,j}(c_1(s(t)))=-\frac{\ddot{s}(t)}{\dot{s}(t)^2}
\dot{c}_1^k(s(t)).
$$

%

We define $f:\Omega \to \R$ 
$$
f(q,v)=\frac{\ddot s_v (t)}{\dot{s}_v(t)^2}, \hbox{ if } v =\dot{\gamma}(0) \hbox{ for some } (\gamma,0) \in \mathcal{C}(q).
$$
Above  $s_v$ is such a reparametrization of $\gamma$ that, $s_v(0)=0$, 
\\
$\frac{d}{dt}\gamma(s(t))|_{t=0}=v$ and  \eqref{eq:geodesic_eq_for_g1_g2} is valid, when $c_1$ is replaced with $\gamma$ and $c_2$ is replaced with $\gamma \circ s_v$. Note that function $f$  is well defined and satisfies the equation \eqref{homogeneous of degree 1}, since geodesic equation \eqref{geodesic equation 1} is preserved under affine re-parametrizations. 
Therefore it holds that for any $q \in U$ the pairs $(\Gamma,0)$ and $(\widetilde{\Gamma},f)$ both solve 
\mtext{the system \eqref{geodesic equation 3} 
for all such coefficients $\frac{d\gamma}{ds}(s)|_{s=s_q}\in \Omega_q$ and $\frac{d^2\gamma}{ds^2}(s)|_{s=s_q}$ that $(\gamma,s_q) \in \mathcal{C}(q)$.} 
The claim follows then from Lemma \ref{Solutions of geodesic eq are gauge equivalent}.
}\end{proof}
\begin{Le}
Suppose that the connections $\Gamma$ and $\widetilde{\Gamma}$ corresponding to metric
tensors $g$ and $\widetilde g$, respectively,
satisfy the equation \eqref{Gauge freedom of connections} with a 1-form $\varphi$.  Then the metric tensors $g$ and $\widetilde{g}$ are geodesically equivalent.
\label{metrics are geodesic equiv}
\end{Le}
\begin{proof}
Let $\gamma(t)$ be a geodesic with respect to the metric  $g$.
Then $\gamma$ satisfies the geodesic equation \eqref{geodesic equation 1}. Substitute $\Gamma$ with $\widetilde{\Gamma}$ into \eqref{geodesic equation 1} to get the equation
$$
\frac{d^2\gamma^k}{dt^2}(t) +\widetilde{\Gamma}^k_{i,j}(\gamma(t))\frac{d\gamma^i}{dt}(t)\frac{d\gamma^j}{dt}(t)=2\frac{d\gamma^k}{dt}(t)\varphi\bigg(\frac{d\gamma}{dt}(t)\bigg).
$$
Write $\kappa(t)=2\varphi (\dot{\gamma}(t))$ and use Lemma \ref{Equivalence of geodesic formulas} to show that 
 there exists a change of parameters $s\mapsto t(s)$ such that $s\mapsto \gamma(t(s))$ is a geodesic with respect to the metric $\widetilde{\Gamma}$. As the roles of $g$ and $\widetilde g$ can be exchanged, the claim follows.
 \end{proof}
By the Lemma \ref{metrics are geodesic equiv},  the equivalence of the distance difference data \eqref{Data1a}-\eqref{Data1b} implies the geodesic equivalence of metric tensors $g$ and $\Psi_{\ast}g_2$ on $N_1$. In the following theorem, that shows that metric tensors $g$ and $\Psi_{\ast}g_2$ coincide also \mtext{in $N_1$,} we will use the implications of the Matveev-Topalov theorem \cite{MaTo}. \mmtext{Their result is also concerned in the appendix of the extended preprint version
of this paper \cite{LaSa}} and its generalizations have
been considered in \cite{BT,Top}.

\begin{Le}
Suppose that manifold $N$ satisfies assumptions of Section \ref{The problem setting of this paper} and $g$ and $\widetilde{g}$ are two metric tensors on $N$. Suppose that these metrics are geodesically equivalent on manifold $N$ and coincide in set $F^{int}\neq \emptyset$. Then $g=\widetilde{g}$ in whole $N$.
\label{geodesic eq -> metrics are the same}
\end{Le}
\begin{proof}
\mmtext{Define a smooth mapping $I_0:TN \rightarrow \R$ as
\begin{equation}
I_0((x,v))=\bigg(\frac{\det (g_x) }{\det(\widetilde{g}_x)}\bigg)^{\frac{2}{n+1}} \widetilde{g}_x(v,v), \label{Matveev formula}
\end{equation}
where  $\widetilde {g}_x(v,v)=\widetilde{g}_{jk}(x)v^jv^k$. 
Note that the function $x\mapsto \frac{\text{det} (g_x) }{\text{det}(\widetilde{g}_x)}$ is coordinate invariant. 

Let $\gamma_{g}$ be a geodesic of metric $g$. Define a smooth path $\beta$ in $TN$ as $\beta(t)=(\gamma_{g}(t), \dot{\gamma}_{g}(t))$, i.e., $\beta$ is an integral curve of the geodesic flow of metric $g$.
The 
 Matveev-Topalov theorem \cite{MaTo} states that if $g$ and $\widetilde{g}$ are geodesically equivalent, then  there are several invariants related to the $(1,1)$-tensor $G=g^{-1}\tilde g$, given in local coordinates by $G^j_k(x)=g^{ji}(x)\tilde g_{ik}(x)$, that are
 constants along integral curves $\beta(t)$. In particular, the
function $t\mapsto I_0(\beta(t))$ is a constant.

A corollary of  Matveev-Topalov theorem, 
\cite[Cor.\ 2]{MaTo} (see also \cite[Cor. 2]{MaTo2} and \cite[Thm.\ 3]{BT}), is that the number $n(x)$ of the 
 different eigenvalues of the map $G(x):T_xN\to T_xN$ is constant 
 at almost every point $x\in N$. Since $G(x)=I$ for $x\in F^{int}$, we have that
 $n(x)=0$ in the set $F^{int}$ having a positive measure. This implies that $n(x)=0$ for almost all $x\in N$.
 Hence  for almost all $x\in N$ there is $c(x)\in \R_+$ such that
 we have $G(x)=c(x)I$, so that
$\tilde g_{ik}(x)=c(x) g_{ik}(x).$
 As $G$  is continuous,
 this holds for all $x\in N$. Summarizing,
 the first implication of the  Matveev-Topalov theorem is that
 $g$  and $\tilde g$ are conformal on the whole manifold $N$.

Let $x_0$ be a point of $N$. Since we assumed that metrics $g$ and $\widetilde{g}$ coincide in set $F$, we have for any  point $z \in F$ and vector $v \in T_zN$ that formula \eqref{Matveev formula} has form
\begin{equation}
I_0(z,v)=\widetilde{g}_z(v,v)=g_z(v,v). \label{I_0=g tilde}
\end{equation}
Let $\gamma(t):=\gamma^g_{z, \xi}(t),\: \xi \in S_zN,$ $z\in F$  be a 
$g$-geodesic passing through $x_0$ such that $x_0=\gamma(t_0)$ for some $t_0\geq 0$. The $I_0((z,\xi))=1$ 
and  by the  Matveev-Topalov theorem,
$I_0$ is constant along the integral curves of geodesic flow of $g$. Thus,
we have
\begin{equation}
I_0(x_0,\dot{\gamma}(t_0))=I_0(z,\xi)=1. \label{I_0=1 on W}
\end{equation}
Define $W_{x_0}$ to be the set of all $g$-unit vectors of $T_{x_0}N$ with respect to metric $g$, such that every vector in $W_{x_0}$ is a velocity vector of some $g$-geodesic starting from $F$ and passing trough $x_0$. Recall that set $W_{x_0}^{int}\subset S_{x_0}N$ is not empty. 

Let $X=(x^1, \ldots ,x^n)$ be any coordinate chart at $x_0$. Formula \eqref{I_0=1 on W} shows that for every $\xi \in W_{x_0}$ we have 
\begin{equation}
g_{ij}(x_0)\xi^i\xi^j=1=I_0(x_0,\xi)=\bigg(\frac{\det (g_{x_0}) }{\det(\widetilde{g}_{x_0})}\bigg)^{\frac{2}{n+1}} \widetilde{g}_{ij}(x_0)\xi^i\xi^j. \label{metrics are conformal}
\end{equation}

Consider an open cone 
$$
W_{x_0}^{int}\cdot \R_+:= \{tw \in T_{x_0}N: t>0, w\in W_{x_0}^{int}\}.
$$
Then the equation \eqref{metrics are conformal} holds for all $\xi \in W_{x_0}^{int}\cdot \R_+$ and since the set $W_{x_0}^{int}\cdot \R_+$ is open and both sides of equation \eqref{metrics are conformal} are smooth in $\xi$, we obtain the equation
\begin{equation}
g_{ij}(x_0)=\bigg(\frac{\det (g_{x_0}) }{\det(\widetilde{g}_{x_0})}\bigg)^{\frac{2}{n+1}} \widetilde{g}_{ij}(x_0), \textrm{ for all } i,j\in \{1,\ldots,n\},
\end{equation} 
as a second order derivative with respect to $\xi$ of equation \eqref{metrics are conformal}.

Denote $f(p):=\frac{\text{det} (g(p)) }{\text{det}(\widetilde{g}(p))}$. Then the above yields 
\begin{equation}
(f(x_0))^{\frac{2}{n+1}}\;\widetilde{g}_{jk}(x_0)=g_{jk}(x_0), \textrm{ for all } j,k \in \{1, \ldots , n\}. \label{local metrics are conformal eq}
\end{equation}
Taking determinants of both sides of 
\eqref{local metrics are conformal eq} 
we see that
\begin{equation}
(f(x_0))^{\frac{2n}{n+1}-1}=1.
 \label{f(x_0)=1}
\end{equation}
Since} we we have assumed the dimension of manifold $N$ is at least 2, we see from equation \eqref{f(x_0)=1} that $f(x_0)=1$. By formula \eqref{local metrics are conformal eq} this implies $g= \widetilde{g}$ on  $M$. 
\end{proof}
Theorem \ref{main theorem} follows now from Theorems \ref{Homeo of m1 and m2} and \ref{Diffeo of m1 and m2} and Lemmas \ref{metrics are geodesic equiv} and 
\ref{geodesic eq -> metrics are the same}.\hfill$\square$\medskip

\section {Application for an inverse problem for a wave equation} 
\label{subset: application}
\matti{Here we consider  an application of Theorem \ref{main theorem} for an inverse problem for a wave equation 
with spontaneous point sources.}
\subsubsection{Support sets of waves produced by point sources}
Let $(N,g)$ be a closed Riemannian manifold. Denote the Laplace-Beltrami operator of metric $g$ by $\Delta_g$. 
We consider a wave equation
\begin{equation}
\left\{ \begin{array}{l}
(\partial_t^2-\Delta_g)G(\cdot,\cdot,y,s)=\kappa(y,s)\delta_{y,s}(\cdot,\cdot), \quad \hbox{in }\cN
\\
G(x,t,y,s)=0, \quad\hbox{for }t<s, \: x\in N.
\end{array} \right.
\label{wave equ}
\end{equation} 
where $\cN=N\times \R$  is the space-time.
The solution $G(x,t,y,s)$ is the wave produced by a point source located at the point $y\in M$ and time $s\in \R$ having the magnitude   $\kappa(y,s)\in \R\setminus \{0\}.$ Above, we have $\delta_{y,s}(x,t)=\delta_y(x)\delta_s(t)$ corresponds 
to a point source at $ (y,s)\in \cN$.

\subsubsection{Inverse coefficient problem with spontaneous point source data}
\matti{Assume that there are two manifolds $(N_1,g_1)$ and $(N_2,g_2)$ satisfying the assumptions given} in Section \ref{The problem setting of this paper} and  
\beq
\label{wave data of manifolds1}
& &\textrm{There exists an isometry } \phi:F_1 \rightarrow F_2
\\ \label{wave data of manifolds2}
& &W_1=W_2
\eeq
where $W_1$  and $W_2$  are collections of supports of waves produced by point sources
taking place at unknown points at unknown time, that is,
$$
W_1=\{\supp(G^1(\cdot,\cdot,y_1,s_1)) \cap (F_1\times \R)\,; \: y_1\in M_1,\ s_1\in \R \}\subset 2^{F_1\times \R}
$$
and
$$
W_2=\{\supp(G^2(\phi(\cdot),\cdot, y_2,s_2)) \cap (F_1\times \R)\,; \: y_2\in M_2,\ s_2\in \R\} \subset 2^{F_1\times \R}
$$
where functions $G^j,$ $j=\{1,2\}$ solve equation \eqref{wave equ} on manifold $N_j$. 
Here $2^{F_j \times \R}=\{V;\ V\subset F_j\times \R\}$ is the power set of $F_j\times \R$.
{Roughly speaking, $W_j$ corresponds to the data that one makes by
observing, in the set $F_j$, the waves that are produced by spontaneous point sources that
that go off, at an unknown time and at an unknown location, in the set $M_j$.

Earlier, the inverse problem for the  sources that are delta-distributions
in time and localized also in the space has been studied in 
\cite{dHT} in the case when the metric $g$ is known. 
Theorem \ref{main theorem} yields the following result telling that the metric $g$  can
be determined when  a large number of waves produced by the point sources are observed}:

\begin{proposition} Let $(N_j,g_j)$, $j=1,2$ be a closed compact Riemannian $n$-manifolds, $n\geq 2$ and $M_j\subset N_j$
be an open set such that $F_j=N_j\setminus M_j$  have non-empty interior.
If the spontaneous point source data of these manifolds coincide,
that is, we have \eqref{wave data of manifolds1}-\eqref{wave data of manifolds2},
then $(N_1,g_1)$ and $(N_2,g_2)$ are isometric.
\end{proposition}

\begin{proof} Let us again omit the sub-indexes of $N,M$, and $F$.
For $y\in M$,  $s\in \R$, and $z\in F$ we define a number
\ba
\mathcal T_{y,s}(z)&=&\sup \{t\in \R;\hbox{ the point $(z,t)$ has a neighborhood}\\
& &\quad \quad \hbox{ $U\subset \cN$ such that $ G(\cdot,\cdot,y,s)\big|_U =0$}\}
\ea
which tells us, what is the first time when the wave $G(\cdot,\cdot,y,s)$ is observed near the
 point $z$. {Using the finite velocity of the wave propagation for the wave equation, see  \cite{Hormander4}, we see that the support of 
$G(\cdot,\cdot,y,s)$ is contained in the future light cone of the point $q=(y,s)\in\cN$ given by
\ba
J^+(q)=\{(y',s')\in \cN;\ s'\ge  d(y',y)+s\}.
\ea
Next, for $\xi=\xi^j\frac \p{\p x^j}\in T_yN$  we denote the corresponding
co-vector by $\xi^\flat=g_{jk}(y)\xi^jdx^j$.
Then the results of  \cite{Hormander2} and \cite{GU1} on the propagation
of singularities for the real principal type operators, in particular for the wave operator, imply that in the set $\mathcal N\setminus \{q\}$  Green's function $G(\cdot,\cdot,y,s)$
 is a Lagrangian distribution
associated to the Lagranian sub-manifold 
\ba
\Sigma_0=\{(\gamma_{y,\eta}(t),s+t;\dot \gamma_{y,\eta}(t)^\flat,dt)\in T^*\cN;\ \eta \in S_yN, \: t> 0\}
\ea
and its principal symbol on $\Sigma_0$ is non-zero.
In particular, \cite[Prop. 2.1]{GU1} implies that $\Sigma=\Sigma_0\cup (T^*_q\mathcal N,\setminus \{0\})$
coincides with the wave front set WF$(u)$ of the solution $u=G(\cdot,\cdot,y,s)$.
This means that a wave emanating from a point source $(y,s)$ propagates along the geodesics of manifold $(N,g)$. 
The image  of WF$(u)$  in the projection $\pi:T^*\cN\to  \cN$ 
coincides the singular support of $u$. Hence, we see that
\beq\label{supports}
& &\textrm{singsupp}\big(G(\cdot,\cdot,y,s)\big)=S(q),\quad \hbox{where}\\ \nonumber
& &S(q)=\{(\exp_y(t\eta),s+t)\in \cN; \eta \in S_yN, \: t\geq 0\}.
\hspace{-10mm}
\eeq
Since the Riemannian manifold $N$ is complete, the space-time $\cN$  is a globally hyperbolic
Lorentzian manifold and we have $\p J^+(q)=S(q)$, see \cite{ONeill}. Summarizing, the above
implies that 
the function $G(\cdot,\cdot,y,s)$ vanishes outside $J^+(q)$ and is non-smooth,
and thus it is non-vanishing in a neighborhood of arbitrary point of $\p J^+(q)$. Thus,
for $z\in F$ we have $\mathcal T_{y,s}(z)=d(z,y)-s$.}
Hence the  distance difference functions satisfy equation
\begin{equation}
D_y(z_1,z_2)=\mathcal T_{y,s}(z_1)-\mathcal T_{y,s}(z_2).
\label{Relation of wave data and DDD}
\end{equation}
Thus, when formulas \eqref{wave data of manifolds1}-\eqref{wave data of manifolds2} are valid,
we see 
using equation \eqref{Relation of wave data and DDD}
that  the distance difference data of the manifolds $N_1$ and $N_2$  coincide,
that is, we have \eqref{Data1a}-\eqref{Data1b}. Hence, the claim follows from Theorem \ref{main theorem}. 
\end{proof}

\matti{Finally, we note that sets $W_j$ are closely related to the light-observation
sets studied in  \cite{preprint} in the study of the inverse problems for non-linear
hyperbolic problems with a time-dependent metric. 
The  light-observation
set $P_U(q)$ corresponding to a source point $q=(y,s)$ and
the observation set $U$ is the intersection of $U$ and the future
light cone emanating from $q$. In fact,
the formula (\ref{supports}) implies that in
the space time $\cN=N\times \R$ the sets $W_j$ coincide with 
the light-observation
sets $P_U(q)$ corresponding to a source point $q=(y,s)$ and
the observation set $U=F\times \R$.}

\medskip

\noindent{\bf Acknowledgements.}
 The authors express their gratitude to Institut Henri Poincare, where part of this work was done. The authors thank prof. V. Matveev and P. Topalov for valuable comments and prof. M. de Hoop for pointing out the relation of geometric inverse problems and the geophysical measurements of microearthquakes.

 ML is partially supported by the Finnish Centre of
Excellence in Inverse Problems Research 2012-2017 and an Academy Professorship.
TS is partially supported by Academy of Finland, project 263235.

\section*{Appendix A: Extensions of data}
\label{Appendix A}

Assume that we are given the set $F=N\setminus M$ and the metric $g|_F$,
but instead of the function $D_{x}: F \times F \rightarrow \R$  we know only
its restriction on the boundary $\p F=\p M$, that is, the map
$$
D_{x}|_{\p F \times \p F}: \p F \times \p F \rightarrow \R, \: D_{x}|_{\p F \times \p F}(z_1,z_2):=d_N(z_1,x)-d_N(z_2,x).
$$

\begin{Le}
The manifold $F=N\setminus M$, the metric $g|_F$,
and the restriction $D_{x}|_{\p F \times \p F}$ of the distance
difference function corresponding to  $x\in M$  determine the distance
difference function $D_{x}: F \times F \rightarrow \R$.
\end{Le}

\begin{proof}
We can determine the map $D_{x}: F \times F \rightarrow \R$
by the formula
$$
 D_{x}(z_1,z_2)=\inf_\a \sup_\b  \bigg(\L(\a)+ D_{x}|_{\p F \times \p F}(\a(1),\beta(1))-
 \L(\beta)\bigg),
  $$
where the infimum is taken over the smooth curves $\a:[0,1]\to F$
from $z_1$  to $\a(1)\in \p F$ and the 
supremum  is taken over the smooth curves $\b:[0,1]\to F$
from $z_2$  to $\beta(1)\in \p F$. 
\end{proof}

This raises the question, if the manifold $(N,g)$ can be reconstructed
when we are given a submanifold of codimension 1, e.g.\ the boundary of the open set $M$ considered
above, and the distance difference functions on this submanifold.
To consider this, assume that we are given a  submanifold    $\widetilde F\subset N$ of dimension $(n-1)$, the metric 
 $g|_{\tilde F}$ on $\tilde F$,
and the collection
$$\{D^x_{\tilde F,N};\ x\in N\}\subset C(\tilde F \times \tilde F),$$ 
where $D^x_{\tilde F,N}(z_1,z_2)=d_N(x,z_1)-d_N(x,z_2)$ for $z_1,z_2\in \tilde F$.
The following counterexample shows that such data do
not uniquely determine the isometry type of $(N,g)$.

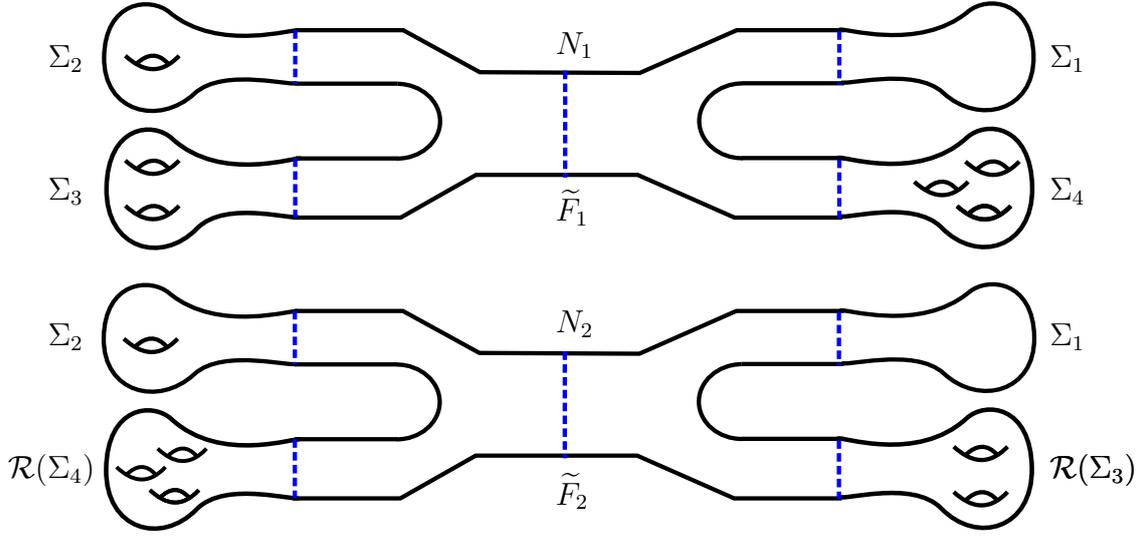
\begin{figure}[h]
 \begin{picture}(350,225)
  \put(0,10){\includegraphics[width=350pt]{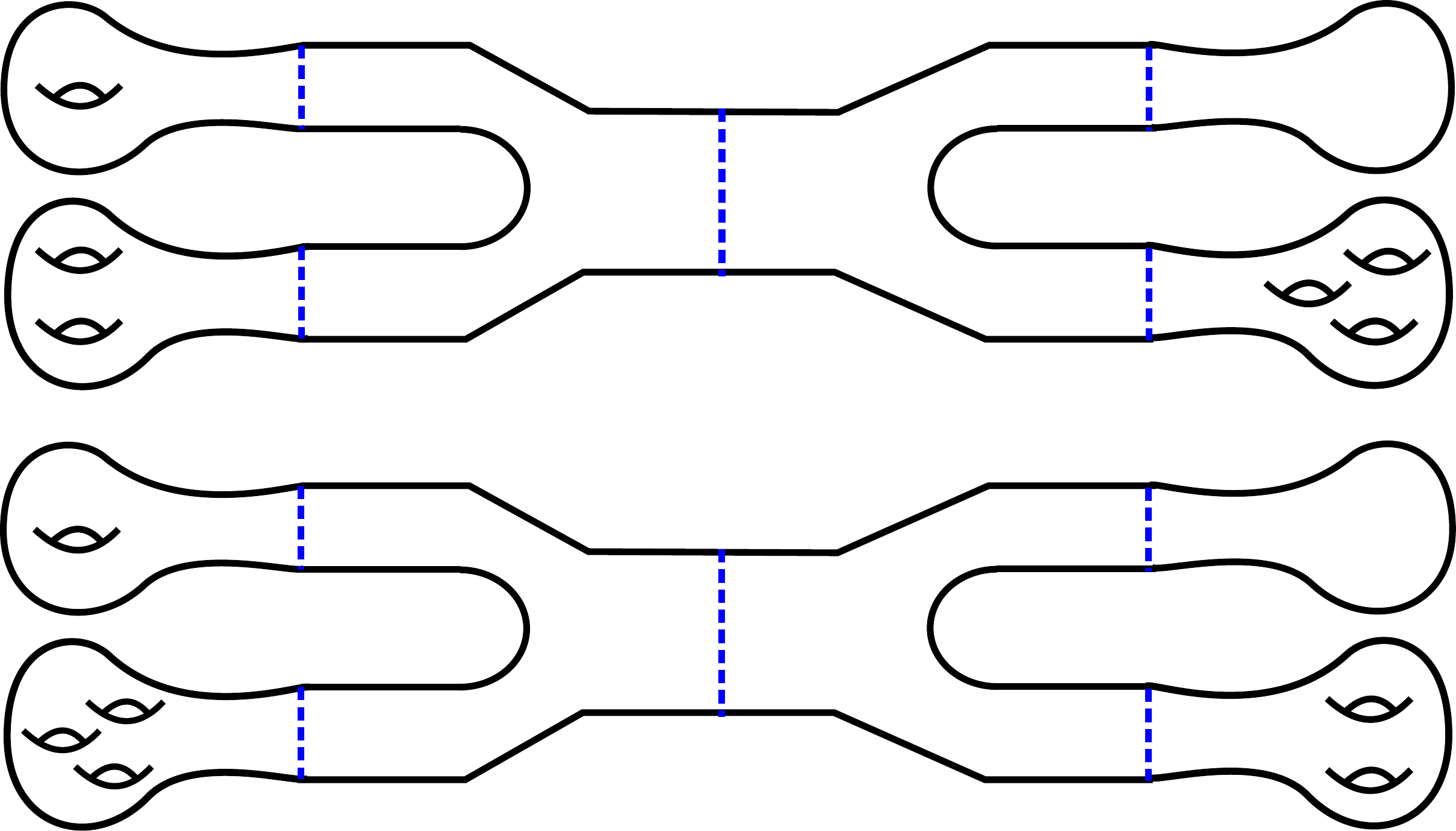}}
  \put(355,185){$\Sigma_1$}
  \put(-20,185){$\Sigma_2$}
  \put(-20,135){$\Sigma_3$}
  \put(355,135){$\Sigma_4$}
  \put(355,80){$\Sigma_1$}
  \put(-20,80){$\Sigma_2$}
  \put(-35,30){$\mathcal R(\Sigma_4)$}
  \put(355,30){$\mathcal R(\Sigma_3)$}
  \put(170,127){$\widetilde F_1$}
  \put(170,20){$\widetilde F_2$}
  \put(355,30){$\mathcal R(\Sigma_3)$}
  \put(170,190){$N_1$}
  \put(170,85){$N_2$}
\end{picture}
\caption{An illustration of manifolds $N_1$ and $N_2$ in Example $A1$.
When  $(n-1)$-dimensional submanifolds $\tilde F_1=\tilde F_2=\tilde F$  are identified,  the distance difference functions $\{D^x_{\tilde F,N_1};\ x\in N_1\}$ and 
$\{D^x_{\tilde F,N_2};\ x\in N_2\}$ coincide.}
\end{figure}

\noindent {\bf Example A1.}
Let $C_r(y)=\{(x_1,x_2)\in \R^2;\ |x_1-y_1|^2+|x_2-y_2|^2=r^2\}$   be a circle of radius $r$
centered at $y=(y_1,y_2)$. Let $p_1=(2,0)$, $p_2=(-2,0)$, $L>3$, and
\ba
& &S_0= C_1(0)\times [-1,1],\\
& &S_1= C_1(p_1)\times [2,L],\\
& &S_2= C_1(p_2)\times [2,L],
\ea 
and $K\subset \R^2\times [1,2]$  be a 2-dimensional surface which
boundary has three components, $C_1(0)\times \{1\}$,  $C_1(p_1)\times \{2\}$,
and  $C_1(p_2)\times \{2\}$, such that the union $S_0\cup K\cup S_1\cup S_2$
is a smooth surface in $\R^3$. Moreover, let $\mathcal R:(x_1,x_2,x_3)\mapsto 
(x_1,x_2,-x_3)$ denote the reflection in the $x_3$-variable. 
Observe that then $\mathcal R(S_0)=S_0$.
We 
define a smooth surface
$$
\Sigma_0=S_0\cup K\cup S_1\cup S_2 \cup \mathcal R(K)\cup 
\mathcal R(S_1)\cup \mathcal R(S_2).
$$
The boundary of $\Sigma_0$ consists of 4 circles,
namely $\Gamma_1=C_1(p_1)\times \{L\}$,
 $\Gamma_2=C_1(p_1)\times \{-L\}$,  $\Gamma_3=C_1(p_2)\times \{L\}$,
 and  $\Gamma_4=C_1(p_2)\times \{-L\}$. Let us consider
 four embedded Riemannian surfaces $\Sigma_j\subset \R^3$, $j=1,2,3,4$,   with boundaries
 $\p \Sigma_j$ are equal to $\Gamma_j$. Assume that near $\p \Sigma_j$ the surfaces
 $\Sigma_j$ are isometric to the Cartesian  product of $\Gamma_j$  and an interval $[0,\e]$   with $\e>0$,
 and that the genus of $\Sigma_j$ is equal to $(j-1)$. Also, assume that 
 $\Sigma_j\cap \Sigma_k=\emptyset$ for $j,k=1,2,3,4$ and
 $\Sigma_0\cap \Sigma_j=\Gamma_j$ for $j=1,2,3,4$.

First, let us construct a manifold $N_1$ by gluing
 surfaces $\Sigma_0$   with $\Sigma_1,\Sigma_2,\Sigma_3,$   and $\Sigma_4$
 such that the boundaries $\Gamma_j$   are glued with $\p \Sigma_j$, $j\in \{1,2,3,4\}$.
 
Second, we construct a manifold $N_2$ by gluing
 surfaces $\Sigma_0$   with $\Sigma_1,\Sigma_2,\mathcal{R}(\Sigma_3),$   and $\mathcal{R}(\Sigma_4)$
 such that the boundaries $\Gamma_j$   are glued with $\p \Sigma_j$ with $j\in \{1,2\}$
 but $\Gamma_3$ is  glued with $\mathcal{R}(\p \Sigma_4)$ and  
  $\Gamma_4$ is  glued with $\mathcal{R}(\p \Sigma_3)$, see Fig.\ 5.  For both manifolds $N_1$ and $N_2$ we give the induced Riemannian metric from $\R^3$. 
Let   $ \tilde F= \tilde F_1= \tilde F_2=S_0\cap (\R^2\times \{0\})$.

  Let us assume that $L$ above is larger than $\diam(K)+10$.
  Then on $N_\ell$, $\ell=1,2$ a minimizing geodesic from $x\in \Sigma_j$, $j\ge 1$
  to $z\in \tilde F$ does not intersect the other sets $\Sigma_k$ with $k\in \{1,2,3,4\}\setminus \{j\}$. 
  Using this we see that the sets
  $\{D^{x}_{\tilde F,N_\ell};\ x\in N_\ell\}\subset C(\tilde F\times \tilde F)$ are the same for $\ell=1,2$.
  As the manifolds $N_1$ and $N_2$   are not isometric,
this implies that the data $(\tilde F,g|_{\tilde F})$   and 
  $\{D^{x}_{\tilde F,N};\ x\in N\}$ do not determine uniquely the manifold
  $(N,g)$.



\begin{thebibliography}{9} 
%
\bibitem{Ammari1} 
H. Ammari, E. Bossy,  V. Jugnon, H. Kang:
Mathematical modeling in photoacoustic imaging of small absorbers,
\textit{SIAM Rev.} 52 (2010), 677-695. 

\bibitem{AKKLT} 
M. Anderson, A. Katsuda, Y. Kurylev, M. Lassas, M. Taylor: Boundary regularity for the Ricci equation, Geometric Convergence, and Gel'fand's Inverse Boundary Problem,  {\it Inventiones Mathematicae} 158 (2004), 261-321.

\bibitem{quake1} 
 B. Artman, I. Podladtchikov, B. Witten: Source location using timereverse
imaging. {\it Geophysical Prospecting} 58(2010), 861--873.
 
%
%
%
%

\bibitem{Bercoff}
J. Bercoff, M. Tanter,  M. Fink: Supersonic shear imaging: a new
technique for soft tissue elasticity mapping, {\it IEEE Trans. Ultrason. Ferroelectr.
Freq. Control}, 51 (2004), 396-409. 

\bibitem{Bal1}
 G. Bal, T. Zhou: Hybrid inverse problems for a system of Maxwell's equations, \textit{Inverse Problems} 30 (2014), no. 5, 055013, 17.

\bibitem{Bal2}
G. Bal, G.  Uhlmann: Inverse diffusion theory of photoacoustics. \textit{Inverse Problems} 26 (2010), no. 8, 085010, 20.

\bibitem{Bal3}
G. Bal: 
 Hybrid inverse problems and internal functionals, Inside Out II, MSRI Publications, Ed.\ by G. Uhlmann, Cambridge Univ., (2012) 
\bibitem{Bal4}
G. Bal, R.Kui: Multi-source quantitative photoacoustic tomography in a diffusive regime. Inverse Problems 27 (2011), no. 7, 075003, 20.

\bibitem{BelKur}
M. Belishev, Y. Kurylev: To the reconstruction of a Riemannian manifold via its spectral data (BC-method).  
{\it Comm. PDE}  17  (1992),   767-804.

\bibitem{Besson}
P. Berard, G. Besson, S. Gallot:  Embedding Riemannian manifolds
by their heat kernel.
  {\it Geom. Funct. Anal.}   4 (1994), 373--398



\bibitem{BT} A. Bolsinov, V. Matveev:
 Geometrical interpretation of Benenti systems,  {\it J.  Geom. Phys.} 44(2003), 489-506.

\bibitem{Chavel} I. Chavel: Riemannian geometry A Modern Introduction, 2nd Edition, Cambridge Univ. (2006)

\bibitem{CS2}
M. Choulli, P. Stefanov: Inverse scattering and inverse boundary value problems for the linear Boltzmann equation.
{\it Comm. PDE} {21} (1996), 763-785. 

%


\bibitem{deHoop1}
M. de Hoop, S. Holman, E. Iversen, M. Lassas, B. Ursin:
Recovering the isometry type of a Riemannian manifold from local boundary diffraction travel times, 
 To appear in \textit{Journal de Math\'ematiques Pures et Appliqu\'ees}.   
%

\bibitem{deHoop2}
M. de Hoop, S. Holman, E. Iversen, M. Lassas, B. Ursin:
Reconstruction of a conformally Euclidean metric from local boundary diffraction travel times,
  {\it SIAM Journal on Applied Mathematics}  46 (2014), 3705-3726. 

\bibitem{dHT}
M. de Hoop, J. Tittelfitz: An inverse source problem for a variable speed
wave equation with discrete-in-time sources. {\it Inverse Problems}.

%


\bibitem{IKL}
H. Isozaki, Y. Kurylev, M. Lassas: Conic singularities, generalized scattering matrix, and inverse scattering on asymptotically hyperbolic surfaces.
{\it J. reine angew. Math.} 

\bibitem{Hormander2}
J. Duistermaat, L. H\"ormander: Fourier integral operators. II. {\it Acta Math.} 128 (1972), 183-269.

\bibitem{GU1}
A. Greenleaf, G., Uhlmann:
 Recovering singularities of a potential from singularities 
of scattering data.  
{\it Comm. Math. Phys.} {157}  (1993), 549--572.



\bibitem{helga}
S. Helgason: Differential geometry and symmetric spaces, \textit{Elsevier}, (1962)

\bibitem{Hormander4}
L. H\"ormander:  The Analysis of Linear Partial Differential Operators IV: Fourier Integral Operators, Springer-Verlag Berlin Heidelberg, 2009
 
\bibitem{Hoskins}
P. Hoskins:
Principles of ultrasound elastography,
\textit{Ultrasound} 20 (2012), 8-15. 

\bibitem{Jeong} 
W. Jeong, H. Lim, H.  Lee, J. Jo, Y. Kim:
Principles and clinical application of ultrasound elastography for diffuse liver disease,
\textit{Ultrasonography} 33 (2014), 149-160.  

\bibitem{quake2}
H. Kao, S.-J. Shan. The source-scanning algorithm: Mapping the distribution
of seismic sources in time and space. {\it Geophysical Journal International},
157 (2004), 589--594.

\bibitem{Kayal}
J. Kayal: 
{\it Microearthquake Seismology and Seismotectonics of South Asia}, Springer, 2008, 552 pp.


\bibitem{KaKu2}
A. Katchalov, Y. Kurylev,
   Multidimensional inverse problem with incomplete boundary spectral
data. {\it 
Comm. PDE} 23 (1998), 55-95.



\bibitem{IBSP} A. Katchalov, Y. Kurylev, M. Lassas: Inverse boundary spectral problems, Chapman and Hall (2001).


%


%
%


\bibitem{Klingenberg} W. Klingenberg: Riemannian geometry, Walter De Gruyter, 1982.


\bibitem{KrKL}
K. Krupchyk, Y. Kurylev, M. Lassas: Inverse spectral problems on a closed manifold. \textit{J.  Math. Pures et Appl.} 90 (2008), 42-59.


\bibitem{Ku5}
Y. Kurylev:  Multidimensional Gel'fand inverse problem
and boundary distance map, {\it Inverse Problems Related with
Geometry}, Ed. H.\ Soga  (1997), 1-15.



%

\bibitem{KL2}
Y. Kurylev, M. Lassas: Inverse Problems and Index Formulae for Dirac Operators. \textit{Adv.  Math.} 221 (2009), 170-216. 

\bibitem{KL3}
 Y. Kurylev, M. Lassas, E. Somersalo: Maxwell's equations with a polarization independent wave velocity: Direct and inverse problems, \textit{J.  Math. Pures et Appl.} 86 (2006), 237-270.
 
\bibitem{KLU-ajm} 
Y. Kurylev, M. Lassas, G. Uhlmann: Rigidity of broken geodesic flow and inverse problems, 
{\it Amer. J. Math.} 132 (2010), 529-562.


\bibitem{preprint} Y. Kurylev, M. Lassas, G. Uhlmann: Seeing through spacetime, 67 pp. ArXiv:1405.3386.
%
%
%
%


\bibitem{Oksanen3} 
Y.  Kurylev, L. Oksanen, G. Paternain:
Inverse problems for the connection Laplacian.  	ArXiv:1509.02645.

\bibitem{LO} 
 M. Lassas, L. Oksanen: Inverse problem for the Riemannian wave equation with Dirichlet data and Neumann data on disjoint sets. \textit{Duke Math. J.} 163 (2014), 1071-1103. 

\bibitem{LaSa} 
M. Lassas, T. Saksala: Determination of a Riemannian manifold from the distance difference functions
with an appendix on Matveev-Topalov theorem. (An extended preprint
version of this paper), see \url{http://wiki.helsinki.fi/display/mathstatHenkilokunta/Teemu+Saksal}

\bibitem{LC} 
T. Levi-Civita: Sulle trasformazioni delle equazioni dinamiche, {\it Ann. di Mat.}, 
24(1896), 255--300.

\bibitem{Led} W. Ledermann: A note on skew symmetric determinants, \textit{Proc. Edinburgh Math. Soc.} 36 (1993), 335-338. 

\bibitem{Lee} J. Lee: Introduction to smooth manifolds, Springer (2000)

\bibitem{Lee2} J. Lee: Riemannian manifolds An Introduction to Curvature, Springer (1997) 

\bibitem{Liu}
H. Liu, G. Uhlmann: Determining both sound speed and internal source in thermo- and photo-acoustic tomography, 
\textit{Inverse Problems} 31 (2015), 105005


\bibitem{MaTo} V. Matveev, P. Topalov: Geodesic Equivalence via Integrability, \textit{Geometriae Dedicata} 96 (2003), 91-115.

\bibitem{MaTo2} V. Matveev, P. Topalov: Trajectory equivalence and corresponding integrals, \textit{Regular Chaotic Dynam.} 3(2) (1998), 30-45.

\bibitem{MaTo3} V. Matveev, P. Topalov: Quantum integrability for the Beltrami-Laplace operator as geodesic equivalence, \textit{Math. Z.} 238 (2001), 833-866.

\bibitem{Ma} V. Matveev: Geodesically equivalent metrics in general relativity, \textit{J. Geom. and Phys.} 62 (2012), 675-691.


\bibitem{McDow1} 
S. McDowall, P. Stefanov, A. Tamasan: Gauge equivalence in stationary radiative transport through media with varying index of refraction,\textit{ Inverse Problems and Imaging} 4  (2010), 151-168.

\bibitem{McDow2} S. McDowall: Optical tomography on simple Riemannian surfaces, \textit{Comm. PDE.} 30 (2005), 1379-1400.

\bibitem{McDow3} S. McDowall: An inverse problem for the transport equation in the presence of a Riemannian metric, \textit{Pac. J. Math} 216 (2004), 107-129. 


\bibitem{ONeill} 
B. O'Neill, {\it Semi-Riemannian geometry}. With applications to relativity. Pure and Applied Mathematics, 103. Academic Press, Inc., 1983. xiii+468 pp.
%





\bibitem{Oksanen1} 
L. Oksanen: Solving an inverse problem for the wave equation by using a minimization algorithm and time-reversed measurements, \textit{Inverse Probl. Imaging} 5 (2011), 731-744.

\bibitem{Oksanen2} 
L. Oksanen: Solving an inverse obstacle problem for the wave equation by using the boundary control method, \textit{Inverse Problems} 29 (2013), 035004.




\bibitem{Ophir} 
J. Ophir et al.:
Elastography: ultrasonic estimation and imaging of the elastic properties of tissues.
\textit{Proc. Inst. Mech. Eng. H.} 213 (1999), 203-233.

\bibitem{Patrolia} 
L. Patrolia: Quantitative photoacoustic tomography with variable index of refraction. \textit{Inverse Probl. Imaging} 7 (2013), 253-265. 


\bibitem{Pestov-Uhlmann} 
L.\ Pestov, G. Uhlmann, H. Zhou: An inverse kinematic problem with internal sources, \textit{Inverse Problems} 31, 055006 (2015),  6.

\bibitem{Pe}
P. Petersen, {\it Riemannian geometry}. Springer, 1998. xvi+432
pp.

\bibitem{StU4}
J. Qian, P. Stefanov, , G. Uhlmann and H. Zhao:
An efficient Neumann series-based algorithm for thermoacoustic and photoacoustic tomography with variable sound speed, \textit{SIAM J. Imaging Sci.} 4(3) (2011), 850-883 pp

\bibitem{Sava}
P. Sava, Micro-earthquake monitoring with sparsely-sampled data,
 \textit{Journal of Petroleum Exploration and Production Technology}
1 (2011), 43-49. 

\bibitem{StU1}
P. Stefanov, G. Uhlmann:
Instability of the linearized problem in multiwave tomography of recovery both the source and the speed, \textit{Inverse Probl. and Imaging}, 7(4) (2013), 1367-1377.

\bibitem{StU3}
P. Stefanov, G. Uhlmann:
Recovery of a Source or a Speed with One Measurement and Applications,
\textit{Transactions of AMS} 365 (2013), 5737-5758.



\bibitem{StU1B}
P. Stefanov,  G. Uhlmann:
 	Multi-wave methods via ultrasound. {\it Inside Out II}, MSRI publications, 60:271-324, 2012.

\bibitem{StU2}
P. Stefanov, G. Uhlmann:
Thermoacoustic tomography with variable sound speed, \textit{Inverse Problems} 25 (2009), 075011.


\bibitem{Top}
P. Topalov: Geodesic compatibility and integrability of geodesic flows, 
{\it J. Math. Phys.} 44 (2003), 913--929,


\end{thebibliography}
\end{document}